\documentclass[reqno,a4paper,11pt]{amsart}

\numberwithin{equation}{section}
\usepackage[utf8]{inputenc}
\usepackage{amsmath, amsthm, amssymb, amscd, accents, bm, amsfonts}
\usepackage{mathtools}
\usepackage{url}
\usepackage{mathrsfs,dsfont}
\usepackage{mathtools}
\usepackage{mdwlist}
\usepackage{paralist}
\usepackage{datetime}
\usepackage{hyperref}
\usepackage{colonequals}
\usepackage{soul}
\usepackage[dvipsnames]{xcolor}

\usepackage[sort,nocompress]{cite}
\mathtoolsset{showonlyrefs}

\newtheorem{definition}{Definition}[section]
\newtheorem{theorem}[definition]{Theorem}
\newtheorem{lemma}[definition]{Lemma}
\newtheorem{corollary}[definition]{Corollary}
\newtheorem{proposition}[definition]{Proposition}

\newtheorem{example}[definition]{Example}
\newtheorem{remark}[definition]{Remark}

\def\N{{\mathbb N}}
\def\Z{{\mathbb Z}}
\def\R{{\mathbb R}}
\def\T{{\mathbb T}}
\def\C{{\mathbb C}}

\def\U{{\mathbb U}}

\newcommand{\Cd}{{\C^d}}

\newcommand{\lspan}{{\mathrm{span}}}
\newcommand{\supp}{{\mathrm{supp}}}
\newcommand{\lt}{{L^2(\R)}}

\newcommand{\im}{{\mathrm{Im} \,}}
\newcommand{\ift}{{\mathcal{F}^{-1}}}
\newcommand{\ft}{{\mathcal{F}}}

\newcommand{\cratio}{{\mathrm{CR}}}

\allowdisplaybreaks

\begin{document}

\title[Nonlinear determination under unimodular constraints]{Nonlinear determination and phase retrieval under unimodular constraints}

\author[Lukas Liehr]{Lukas Liehr}
\address{Department of Mathematics, Bar-Ilan University, Ramat-Gan 5290002, Israel}
\email{lukas.liehr@biu.ac.il}

\author[Tomasz Szczepanski]{Tomasz Szczepanski}
\address{Department of Mathematical and Statistical Sciences, University of Alberta, Edmonton, AB, Canada, T6G 2G1}
\email{tszczepa@ualberta.ca}

%%%%%%%%%%%%%%%%%%%%%%%%%%%%%%%%%%%%%%%%%%%%%%%%%%%%%%%%%%%%%%%%%%%%
\date{\today}
%%%%%%%%%%%%%%%%%%%%%%%%%%%%%%%%%%%%%%%%%%%%%%%%%%%%%%%%%%%%%%%%%%%%

\begin{abstract}
We study nonlinear determination problems in Hilbert spaces in which inner products are observed up to prescribed rotations in the complex plane. Given a Hilbert space $H$ and a subset $\Theta$ of the unit circle $\mathbb{T}$, we say that a system $\mathbf{G}\subseteq H$ does \emph{$\Theta$-phase retrieval} ($\Theta$-PR) if for all $f,h\in H$ the condition that for every $g\in\mathbf{G}$ there exists $\theta_g\in\Theta$ with $\langle f,g\rangle=\theta_g\langle h,g\rangle$ forces $f=\theta h$ for some $\theta\in\Theta$. This framework unifies classical phase retrieval ($\Theta=\mathbb{T}$) and sign retrieval ($\Theta=\{1,-1\}$).
For every countable $\Theta$ we give a complete characterization of $\Theta$-PR in terms of covers of $\mathbf{G}$ and geometric relations among vectors in the corresponding orthogonal complements, extending the complement-property characterization of Cahill, Casazza, and Daubechies. For cyclic phase sets we show that $\Theta$-PR is equivalent to the existence of specific second-order recurrence relations. We apply this to obtain a sharp lattice density criterion for $\Theta$-PR of exponential systems.
For uncountable $\Theta$ we obtain a topological dichotomy in the Fourier determination setting, showing that $\Theta$-PR is characterized in terms of connectedness of $\Theta$.
We further develop a Möbius-invariant framework, proving that $\Theta$-PR is preserved under circle automorphisms and is governed by projective invariants such as the cross ratio. Finally, in $\mathbb{C}^d$ we determine sharp impossibility thresholds and prove that for countable $\Theta$ the property is generic once one passes the failure regime, yielding the minimal number of vectors required for $\Theta$-PR.
\end{abstract}

\maketitle

\section{Introduction and main results}

Let $H$ be a complex Hilbert space with inner product $\langle \cdot, \cdot \rangle$ and let $\mathbf{G} \subseteq H$ be a system of vectors. Given $f,h \in H$, assume that for every $g \in \mathbf{G}$ we have
\begin{equation}\label{eq:linear_condition}
\langle f,g \rangle = \langle h,g \rangle.
\end{equation}
If $\mathbf{G}$ is complete in $H$, i.e., the linear span of $\mathbf{G}$ is dense in $H$, then condition \eqref{eq:linear_condition} forces $f=h$. In other words, the collection of scalar products $\{\langle f,g \rangle : g \in \mathbf{G} \}$ uniquely determines the vector $f$. Now consider the situation where these scalar products are not observed exactly, but only up to prescribed rotations in the complex plane. Precisely, suppose that for each $g \in \mathbf{G}$ the inner products $\langle f,g\rangle$ and $\langle h,g\rangle$ are allowed to differ by a rotation factor $\theta_g$ belonging to a fixed subset $\Theta$ of the unit circle $\T \coloneqq \{ z \in \C : |z| = 1 \}$. The linear constraint \eqref{eq:linear_condition} is therefore replaced by the following nonlinear condition: for every $g \in \mathbf{G}$ there exists $\theta_g \in \Theta$ such that
$$
\langle f,g \rangle = \theta_g \langle h,g \rangle.
$$
As before, we would like this family of constraints to determine $f$. In general, we can no longer conclude that $f$ and $h$ coincide; the best one can expect is that there exists $\theta \in \Theta$ such that $f = \theta h$. Moreover, as we will observe, if $\Theta$ has at least two elements then completeness of the system $\mathbf{G}$ alone does not imply that $f= \theta h$ and one has to impose specific redundancy in $\mathbf{G}$.

Two classical special cases fit into this framework. If $\Theta = \T$, then for each $g \in \mathbf{G}$ the existence of some $\theta_g \in \T$ with $\langle f,g \rangle = \theta_g \langle h,g \rangle$ is equivalent to the magnitude equality
\begin{equation}\label{eq:magnitude}
    |\langle f,g \rangle| = | \langle h,g \rangle|.
\end{equation}
In this case, the problem of deciding whether the family of equalities in \eqref{eq:magnitude} implies $f = \theta h$ for some $\theta \in \T$ is known as the uniqueness problem in phase retrieval (PR) (see, for instance, the survey \cite{grohskopp2020phase} and the references therein). This problem was first considered in form of the Pauli problem, which is concerned with the question of whether a function is determined from its absolute value and the absolute value of its Fourier transform \cite{pauli1990allgemeinen,ramos}.

If $H$ is a real Hilbert space, then condition \ref{eq:magnitude} reduces to
\begin{equation}\label{eq:sign}
    \langle f,g \rangle =  \pm \langle h,g \rangle
\end{equation}
and one speaks of the sign retrieval problem, which we will also refer to as real phase retrieval (real PR). Phase retrieval and its real analogue arise naturally in a variety of applications where phase information is lost or unobservable. Prominent examples include X-ray crystallography \cite{elser2003solution,thibault2008high}, diffraction imaging \cite{miao2015beyond,shechtman2015phase}, and quantum mechanics \cite{orl1994phase,raymer1997measuring}. Beyond these applications, phase retrieval has become increasingly relevant in signal processing and data analysis \cite{mallat2016understanding}. From a mathematical perspective, these applications have motivated an extensive mathematical theory concerned with the fundamental questions of uniqueness and stability. These questions have been investigated from a wide range of viewpoints, for instance, functional analytic perspectives \cite{Freeman2024,christ2022examples,freeman2025cahill,aldroubi2020phaseless,alharbi2023stable}, convex geometry \cite{bianchi2011phase}, complex analysis \cite{grohs2021stable,jaming2014uniqueness,wellershoff2024phase}, sampling theory \cite{grohs2025phase,jaming2024gabor,ramos,alaifari2024unique}, finite-dimensional and algebraic perspectives \cite{bendory2022algebraic,amir2025stability}, group-theoretical settings \cite{bartusel2023phase} and algorithmic approaches \cite{candes2015phase,candes2015phase2,steinerberger2022eigenvector}.

In the present paper we consider a unified version of these problems, allowing an arbitrary rotation set $\Theta \subseteq \T$. This leads us to the following definition.

\begin{definition}
Let $H$ be a Hilbert space, let $\mathbf{G} \subseteq H$ and let $\Theta \subseteq \T$. We say that $\mathbf{G}$ does $\Theta$-PR (in $H$) if for every $f,h \in H$ the following holds: if for every $g \in \mathbf{G}$ there exists $\theta_g \in \Theta$ such that
\begin{equation}\label{eq:theta_pr}
    \langle f,g \rangle = \theta_g \langle h,g \rangle,
\end{equation}
then $f=\theta h$ for some $\theta \in \Theta$.
\end{definition}

Classical spanning properties are related to $\Theta$-PR, but none of them alone provide a complete description of the systems $\mathbf{G}$ doing $\Theta$-PR. On the one hand, any system doing $\Theta$-PR must be complete. On the other hand, for finite sets $\Theta$, imposing sufficiently strong redundancy such as overcompleteness, reduces the nonlinear determination problem to a linear one and implies that a system does $\Theta$-PR. However, overcompleteness is neither necessary nor intrinsic to $\Theta$-PR. We refer to Section \ref{sec:completness_and_overcompleteness} for detailed comparisons between $\Theta$-PR and classical spanning properties.

A sharper picture emerges when exactly two phases are allowed, i.e., $|\Theta|=2$, where $|X|$ denotes the cardinality of a set $X$. In this case, $\Theta$-PR is governed by a strengthened form of completeness, known as the complement property: $\mathbf{G} \subseteq H$ has the complement property if for every $\mathbf{S} \subseteq \mathbf{G}$ it holds that $\mathbf{S}$ is complete in $H$ or $\mathbf{G} \setminus \mathbf{S}$ is complete in $H$. We show the following.

\begin{proposition}\label{prop:cp}
Let $\mathbf{G} \subseteq H$ and let $\Theta \subseteq \T$ satisfy $|\Theta|=2$. Then $\mathbf{G}$ does $\Theta$-PR if and only if $\mathbf{G}$ has the complement property.
\end{proposition}

When specified to systems in real Hilbert spaces doing real PR (that is $\Theta$-PR with respect to $\Theta=\{-1,1\}$), then the latter Proposition is a classical result in the literature. In the finite dimensional setting this was shown in \cite{balan2006signal}, in general Hilbert spaces it was proved in \cite{cahill2016phase}, and the setting of Banach spaces was treated in \cite{alaifari2017phase}. Moreover, in these references it is shown that for complex Hilbert spaces the complement property is only a necessary condition for systems to do PR (i.e. $\Theta$-PR with $\Theta=\T$). In our broader framework of $\Theta$-PR, the distinction between real and complex PR disappears. In particular, Proposition \ref{prop:cp} shows that the property of a system doing $\Theta$-PR does not depend on $\Theta\subseteq\T$, whenever $|\Theta|=2$.

In the language of failure, Proposition \ref{prop:cp} allows us to deduce the following equivalent statement: given $\Theta\subseteq \T$ with $|\Theta|=2$, a system $\mathbf{G}$ fails $\Theta$-PR precisely when there exist $\mathbf{G}_1, \mathbf{G}_2 \subseteq \mathbf{G}$ such that $\mathbf{G}_1\cup \mathbf{G}_2=\mathbf{G}$, with both $\mathbf{G}_1$ and $\mathbf{G}_2$ being incomplete in $H$. The failure of $\Theta$-PR for larger rotation sets $\Theta$ is encoded in a more delicate geometrical structure. In particular, obstructions can no longer be detected by mere completeness properties of a partition, but instead arise from a coherent interaction between multiple subsets of $\mathbf{G}$. The next theorem makes this phenomenon precise by characterizing failure of $\Theta$-PR for countable sets $\Theta$ in terms of covers of $\mathbf{G}$ and algebraic relations among vectors in the associated orthogonal complements. It may be viewed as a far-reaching extension of the characterization in Hilbert spaces based on the complement property obtained by Cahill, Casazza, and Daubechies \cite[Theorem 1.2]{cahill2016phase}. We say that $\{\mathbf{G}_j\}_{j \in \N}$ is a cover of $\mathbf{G}$ if for every $j\in\N$ we have $\mathbf{G}_j\subseteq \mathbf{G}$ and $\bigcup_{j=1}^\infty \mathbf{G}_j = \mathbf{G}$. Moreover, we denote the orthogonal complement of a set $X \subseteq H$ by $X^\perp$.

\begin{theorem}\label{thm:m1}
   Let $H$ be a Hilbert space and suppose that $\Theta=\{\theta_j\}_{j \in \N}\subseteq \T$ satisfies $\theta_1\neq \theta_2$. Then $\mathbf{G} \subseteq H$ fails $\Theta$-PR if and only if there exist a cover $\{\mathbf{G}_j\}_{j \in \N}$ of $\mathbf{G}$ and a sequence $\{ x_j \}_{j \in \N} \subseteq H$ satisfying
    \begin{enumerate}
        \item $0 \neq x_j \in \mathbf{G}_j^\perp$ for every $j\in \N$,
        \item $x_1$ and $x_2$ are linearly independent,
        \item for every $j \geq 3$ it holds that
        $$x_j = \frac{\theta_j-\theta_1}{\theta_2-\theta_1}x_2 - \frac{\theta_j-\theta_2}{\theta_2-\theta_1}x_1.$$
    \end{enumerate} 
\end{theorem}

A careful examination of Theorem \ref{thm:m1} will allow us to deduce that in the case when $|\Theta|=3$, similarly to the case of $|\Theta| = 2$, the property of a system doing $\Theta$-PR is independent of the choice of $\Theta$ (see Corollary \ref{cor:3_pr_characterization}). Furthermore, Theorem \ref{thm:m1} yields a particularly transparent description when the rotation set $\Theta$ is generated by a single element. In this case, the relations appearing in Theorem \ref{thm:m1} allow us to view failure of $\Theta$-PR through the lens of existence of a recurrence relation.

\begin{proposition}\label{prop:recurrence}
Let $H$ be a Hilbert space and let $\Theta=\{\theta_j\}_{j \in \N} \subseteq \T$ be defined by $\theta_j=\omega^j$ for some $\omega \in \T\setminus\{1\}$. Then $\mathbf{G} \subseteq H$ fails $\Theta$-PR if and only if there exist a cover $\{\mathbf{G}_j\}_{j \in \N}$ of $\mathbf{G}$ and a sequence $\{ x_j \}_{j \in \N} \subseteq H$ satisfying
\begin{enumerate}
        \item $0 \neq x_j \in \mathbf{G}_j^\perp$ for every $j\in \N$,
        \item $x_1$ and $x_2$ are linearly independent,
        \item for every $j \geq 3$, the vector $x_j$ is given by the second-order recurrence relation $x_j = (1+\omega)x_{j-1} - \omega x_{j-2}$.
    \end{enumerate}
\end{proposition}

Proposition \ref{prop:recurrence} turns out to be a particularly useful tool in the study of Fourier uniqueness problems under nonlinear constraints. Questions of this type are commonly studied under the names phase retrieval and sign retrieval in the Paley–Wiener space \cite{thakur2011reconstruction,alaifari2017reconstructing,pohl2014phaseless,lai2021conjugate,akutowicz1956determination,jaming2022effect,alaifari2024unique,wellershoff2024phase}.

To discuss this setting, define for $\lambda \in \C$ the exponential function $e_\lambda(x)=e^{2\pi i \lambda x}$. Given $\Lambda \subseteq \C$, we denote by $E(\Lambda) = \{ e_\lambda \}_{\lambda \in \Lambda}$ the associated exponential system with frequencies in $\Lambda$. For a function $f \in L^2[0,a]$ with $a>0$, let $c_n(f)$ denote its Fourier coefficients,
$$
c_n(f) \coloneqq \int_0^a f(t) e^{-2\pi i \frac{n}{a} t} \, dt, \quad n \in \Z.
$$
By the classical Fourier uniqueness theorem, if $c_n(f) = 0$ for all $n \in \Z$, then $f = 0$. Equivalently, the exponential system $E(\tfrac{1}{a}\Z)$ is complete in $L^2[0,a]$. More generally, for a lattice $\Lambda = \alpha \Z$, the system $E(\Lambda)$ is complete in $L^2[0,a]$ if and only if its (asymptotic) density
$$
D(\Lambda) = \lim_{r \to \infty} \frac{|\Lambda \cap [-r,r]|}{2r} = \frac{1}{\alpha}
$$
satisfies $D(\Lambda) \ge a$ \cite{Young}. This classical completeness condition can be reformulated as follows: for a lattice $\Lambda$, the implication
\begin{equation}\label{eq:fourier_lin}
    \int_0^a f(t) e^{-2\pi i \lambda t} \, dt = \int_0^a h(t) e^{-2\pi i \lambda t} \, dt, \quad \lambda \in \Lambda \implies f=h
\end{equation}
holds for all $f,h \in L^2[0,a]$ precisely when $D(\Lambda) \ge a$. If instead we ask whether $E(\Lambda)$ does $\Theta$-PR, then the linear condition on the left-hand side of \eqref{eq:fourier_lin} is replaced by the nonlinear constraint
$$
\int_0^a f(t) e^{-2\pi i \lambda t} \, dt = \theta_\lambda \int_0^a h(t) e^{-2\pi i \lambda t} \, dt, \quad \lambda \in \Lambda, \quad \theta_\lambda \in \Theta.
$$
When $\Theta$ consists of roots of unity, Proposition \ref{prop:recurrence} yields a concise characterization of $\Theta$-PR in terms of a density condition on $\Lambda$. This result may be viewed as a nonlinear analogue of the classical Fourier uniqueness theorem. Moreover, it generalizes results by Alaifari, Daubechies, Grohs, and Thakur who studied the case $\Theta = \{1,-1 \}$ \cite{thakur2011reconstruction,alaifari2017reconstructing}.

\begin{theorem}\label{thm:lattice_density_characterization}
    Let $\Lambda \subseteq \R$ be a lattice and let $\Theta = \{ e^{2\pi i k/n} : k=0,\dots,n-1 \}$. Then $E(\Lambda)$ does $\Theta$-PR in $L^2[0,a]$ if and only if $D(\Lambda) \geq na$.
\end{theorem}

The previous theorem implies, in particular, that the full exponential system $E(\mathbb{R})$ does $\Theta$-PR in $L^{2}[0,a]$ for every set $\Theta$ consisting of $n$-th roots of unity. In the special case $\Theta = \{1,-1\}$, this result was  obtained in \cite{thakur2011reconstruction}. On the other
hand, it is known that when $\Theta = \mathbb{T}$, the system $E(\mathbb{R})$ does not do PR, as counterexamples can be constructed using the so-called zero-flipping technique \cite{akutowicz1956determination}. What happens if $\Theta$ is a general set, different from roots of unity or the whole torus $\T$? In the following statement, we provide a topological characterization of all $\Theta$ for which the (full) exponential system $E(\mathbb{R})$ does $\Theta$-PR.

\begin{theorem}\label{thm:ER_disconnected}
    The exponential system $E(\R)$ does $\Theta$-PR in $L^2[0,a]$ if and only if $\Theta$ is totally disconnected.
\end{theorem}

Next, we develop a Möbius-invariant framework for $\Theta$-PR. Concretely, the upcoming statement shows that $\Theta$-PR is invariant under Möbius transforms that preserve the unit circle $\T$. This invariance under Möbius transforms turns out to be a powerful tool in the study of $\Theta$-PR. We denote by $\mathrm{Aut}(\T)$ the group of bijections on $\T$ and by $M(\Theta)$ the image of $\Theta$ under $M \in \mathrm{Aut}(\T)$.

\begin{theorem}\label{thm:mbs}
    Let $H$ be a Hilbert space. Then for every $\mathbf{G} \subseteq H$ and every $\Theta \subseteq \T$, the following statements are equivalent:
    \begin{enumerate}
        \item $\mathbf{G}$ does $\Theta$-PR.
        \item There exists a Möbius transform $M\in\mathrm{Aut}(\T)$ such that $\mathbf{G}$ does $M(\Theta)$-PR.
        \item For every Möbius transform $M\in\mathrm{Aut}(\T)$ it holds that $\mathbf{G}$ does $M(\Theta)$-PR.
    \end{enumerate}
\end{theorem}

As mentioned earlier, the property of a system $\mathbf{G}$ doing $\Theta$-PR does not depend on $\Theta$, whenever $|\Theta|\leq 3$. The above theorem showcases the underlying principle governing this property: it is a classical fact that for any two $\Theta,\Theta'\subseteq \T$ with $|\Theta|=|\Theta'|\leq 3$, there exists a Möbius transform $M\in \text{Aut}(\T)$ such that $M(\Theta)=\Theta'$.
For sets $\Theta$ consisting of four elements the property of doing $\Theta$-PR is governed by projective invariants such as the cross ratio. Recall that the cross ratio of pairwise distinct $z_1,z_2,z_3,z_4 \in \C$ is defined by
$$
\cratio(z_1,z_2;z_3,z_4) = \frac{(z_1-z_3)(z_2-z_4)}{(z_1-z_4)(z_2-z_3)}.
$$
With the help of Theorem \ref{thm:mbs} we are able to deduce that if $\Theta,\Theta'\subseteq \T$ satisfy $|\Theta|=|\Theta'|=4$ and $\cratio(\Theta)=\cratio(\Theta')$, then a system $\mathbf{G}$ does $\Theta$-PR if and only if it does $\Theta'$-PR (see Corollary \ref{cor:4_pr_and_cr}).
On the other hand, we will show in Section \ref{sec:c2_characterization} that in general the property of a system doing $\Theta$-PR is dependent on the choice of $\Theta$, whenever $|\Theta|\geq 4$. In particular there exist $\Theta$,$\Theta'$ and $\mathbf{G}$ with $|\Theta|=|\Theta'|=4$ and $\cratio(\Theta) \neq \cratio(\Theta')$ such that $\mathbf{G}$ does $\Theta$-PR but not $\Theta'$-PR.

The remaining results focus on the finite dimensional setting $H=\C^d$ which constitutes an active and extensively studied research area in phase retrieval \cite{conca2015algebraic,balan2006signal,balan2009painless,xia2025stability}. A central question in this setting is the minimality problem: determine the smallest $m$ for which there exists a system $\mathbf{G}=\{g_j\}_{j=1}^m\subseteq \C^d$ that does phase retrieval in $\Cd$ (or analogously real phase retrieval in $\R^d$). Motivated by this question, we define the minimality quantity $\mathcal{N}(\C^d,\Theta)$ via
$$
\mathcal{N}(\C^d,\Theta) = \min \left \{ m \in \N :  \exists \mathbf{G} \subseteq \C^d \ \text{with} \ |\mathbf{G}| = m \ \text{that does} \ \Theta\text{-PR} \ \text{in} \ \C^d \right \}.
$$
A closely related line of research studies genericity: for a fixed $m$, how large is the collection of systems $\{g_j\}_{j=1}^m$ that do phase retrieval, and does a random or generic choice already guarantee this property? In this section we address these questions for $\Theta$-PR. We begin with necessary conditions, showing that $\Theta$-PR cannot hold below certain thresholds.

\begin{theorem}\label{thm:failure}
    Let $d \geq 2$ and let $\mathbf{G} = \{ g_j \}_{j=1}^m \subseteq \C^d$. 
    \begin{enumerate}
        \item If $m \leq 2d-2$, then $\mathbf{G}$ fails $\Theta$-PR for every $\Theta \subseteq \T$ with $|\Theta| \geq 2.$
        \item If $m=2d-1$, then $\mathbf{G}$ fails $\Theta$-PR for every $\Theta \subseteq \T$ with $|\Theta| \geq 3$.
    \end{enumerate}
\end{theorem}

Theorem \ref{thm:failure} identifies a regime in which $\Theta$-PR is impossible, regardless of the choice of vectors $\{ g_j \}_{j=1}^m$. Beyond this regime, the situation changes noticeably. For instance, it follows from Proposition \ref{prop:cp} that if $|\Theta|=2$ and $m \geq 2d-1$, every full-spark system $\{ g_j \}_{j=1}^m$ does $\Theta$-PR in $\C^d$. Recall that a system is called full-spark if every subset of $d$ vectors is linearly independent. It is known that for fixed $d$ and $m \geq 2d-1$, the collection of full-spark systems is generic in an algebraic sense: it forms a non-empty Zariski-open subset of $\C^{d \times m} \simeq \C^{dm}$ \cite{alexeev2012full} (see Section \ref{sec:zariski_topology} for details on Zariski topology). This implies that $\mathcal{N}(\C^d,\Theta) = 2d-1$. 

The case $\Theta = \T$, corresponding to classical phase retrieval, is substantially more delicate. It is known that the collection of systems $\{ g_j \}_{j=1}^m \subseteq \C^d$ with $m \geq 4d-4$ doing phase retrieval contains a non-empty Zariski-open set \cite{conca2015algebraic}, although the exact minimal number $\mathcal{N}(\C^d,\T)$ remains unknown for general $d \in \N$. Using embedding results for complex projective spaces, it has been shown that the bound $4d-4$ is optimal up to a logarithmic term in the dimension $d$ \cite{heinosaari2013quantum}. However, it is not sharp, i.e., there exists $d$ for which $\mathcal{N}(\C^d,\T) < 4d-4$ \cite{vinzant2015small}.

The following theorem demonstrates that, once the failure regime of Theorem \ref{thm:failure} is exceeded, $\Theta$-PR becomes a generic phenomenon. More precisely, for finite $\Theta$ the property holds on a non-empty Zariski-open set, while for countably infinite $\Theta$ it holds on a dense (in Euclidean topology) $G_\delta$-set of full Lebesgue measure. This demonstrates that in both the finite and countable case, the systems doing $\Theta$-PR is a rich class. In particular, such systems occur with probability one when the vectors are drawn from any continuous probability distribution.

\begin{theorem}\label{thm:generic_2d}
    Let $d\geq 2$, let $m \geq 2d$, and let $\Theta\subseteq \T$. Moreover, let $\mathcal{C} \subseteq \C^{d \times m}$ denote all systems $\{ g_j \}_{j=1}^m \simeq (g_1, \dots, g_m) \in \C^{d \times m}$ doing $\Theta$-PR in $\C^d$.
    \begin{enumerate}
        \item If $\Theta$ is finite then $\mathcal{C}$ contains a non-empty Zariski-open set.
        \item If $\Theta$ is countable then $\mathcal{C}$ contains a dense $G_\delta$-set of full Lebesgue measure.
    \end{enumerate}
\end{theorem}

We note that in the literature, a statement as the first one in Theorem \ref{thm:generic_2d} is commonly described by saying that $\Theta$-PR holds generically, or more precisely, Zariski-generically, meaning that it holds on a non-empty Zariski-open subset of $\C^{d \times m}$. Theorem \ref{thm:failure} and Theorem \ref{thm:generic_2d} imply the following result.

\begin{corollary}\label{cor:minimality_theta_pr}
    Let $d\geq 2$ and let $\Theta\subseteq \T$ be countable. Then
    $$
    \mathcal{N}(\C^d,\Theta) = \begin{cases}
        d, & |\Theta| = 1 \\
        2d-1, & |\Theta| = 2 \\
        2d, & |\Theta| \geq 3
    \end{cases}.
$$
\end{corollary}
This completely resolves the minimality problem for all countable sets $\Theta$. Moreover, for every $m \geq \mathcal{N}(\C^d,\Theta)$, the collection $\mathcal{C}$ of systems with $m$ elements doing $\Theta$-PR is dense in $\C^{d \times m}$. Hence, $\mathcal{C}$ is either the empty set (for $m < \mathcal{N}(\C^d,\Theta)$) or it is dense. Whether a similarly sharp transition occurs in the case $\Theta = \T$ remains an open problem in the literature (see, for example, \cite[Section 2]{cahill2016phase}).

Note that Corollary \ref{cor:minimality_theta_pr} applies to every countable set $\Theta$, including sets that are dense in $\T$. However, the situation changes when $\Theta$ is uncountable. Consider for instance the case when $\Theta$ contains a non-trivial arc of the form $\{ e^{it} : t \in [a,b] \}$ where $a<b$. In this case, $\Theta$-PR in finite dimensions turns out to be equivalent to PR. In particular, the minimality quantity $\mathcal{N}(\C^d,\Theta)$ jumps from $2d$ in the countable case to approximately $4d$.
\begin{proposition}\label{prop:arc}
Suppose that $\Theta \subseteq \T$ contains a non-trivial arc. Then $\mathbf{G}\subseteq \C^d$ does $\Theta$-PR if and only if it does PR. In particular,
$$
\mathcal{N}(\C^d,\Theta) \geq 4d -4 - 2\log_2(d).
$$
\end{proposition}

\textbf{Outline.}
In Section 2 we revisit completeness and overcompleteness, and explain how they interact with $\Theta$-PR. We also treat the two-phase regime $|\Theta|=2$, proving that $\Theta$-PR is equivalent to the complement property (Proposition \ref{prop:cp}).

Section 3 develops the theory for countable phase sets $\Theta$: we prove the general characterization of failure of $\Theta$-PR via covers and orthogonality relations (Theorem \ref{thm:m1}) and derive Proposition \ref{prop:recurrence} that treats cyclic phase sets. We then apply this framework to exponential systems and obtain a sharp lattice density criterion for $\Theta$-PR in $L^2[0,a]$ (Theorem \ref{thm:lattice_density_characterization}). In addition, we provide a proof of Theorem \ref{thm:ER_disconnected}.

In Section 4 we introduce the Möbius-invariant perspective. We prove that $\Theta$-PR is preserved under Möbius transforms mapping $\T$ onto itself (Theorem \ref{thm:mbs}) and relate changes in $\Theta$ to projective invariants such as cross ratios.

Finally, Section 5 focuses on the finite-dimensional setting $H=\C^d$. We first analyze $\Theta$-PR in $\C^2$ in terms of explicit algebraic and geometric conditions. We then address minimality and genericity in higher dimensions: we prove sharp impossibility thresholds (Theorem \ref{thm:failure}), show that $\Theta$-PR becomes generic once $m \geq 2d$ for finite and countable $\Theta$ (Theorem \ref{thm:generic_2d}), and conclude with the behavior for uncountable $\Theta$, including the arc case where $\Theta$-PR reduces to standard PR (Proposition \ref{prop:arc}).

\section{Completeness, overcompleteness, and complement property}\label{sec:2}

\subsection{Completeness and overcompleteness}\label{sec:completness_and_overcompleteness}

Unless specified otherwise, $H$ denotes a complex Hilbert space with $\dim(H)\geq 2$, equipped with a scalar product $\langle \cdot, \cdot \rangle$ that is linear in the first argument and conjugate linear in the second argument. Furthermore, whenever we write $\Theta \subseteq \T$ and $\mathbf{G}\subseteq H$ we assume that $\Theta \neq \varnothing$ and $\mathbf{G}\neq \emptyset$, respectively. For a subset $\mathbf{G} \subseteq H$ we denote by $\lspan(\mathbf{G})$ the linear span of $\mathbf{G}$, consisting of all finite linear combinations of elements in $\mathbf{G}$. Recall that $\mathbf{G}$ is said to be complete in $H$ if the closure of $\lspan(\mathbf{G})$ coincides with $H$. By an application of the Hahn-Banach theorem, $\mathbf{G}$ is complete in $H$ if and only if for every $f \in H$, the condition $\langle f,g \rangle =0$ for all $g \in \mathbf{G}$ forces $f=0$, or equivalently, $\mathbf{G}^\perp=\{0\}$, where $\mathbf{G}^\perp$ denotes the orthogonal complement of set $\mathbf{G}$. The connection between completeness and $\Theta$-PR is summarized in the following observation.

\begin{lemma}\label{lma:completeness}
Let $\Theta\subseteq\T$. If $\mathbf{G} \subseteq H$ does $\Theta$-PR then $\mathbf{G}$ is complete. Furthermore, if $\mathbf{G}$ is complete and $\Theta$ is a singleton, then $\mathbf{G}$ does $\Theta$-PR.
\end{lemma}
\begin{proof}
If $f \in H$ satisfies $\langle f,g\rangle=0$ for every $g\in\mathbf{G}$ then, in particular, 
$$\langle f,g\rangle = \theta_g\langle0,g\rangle$$
for all $g \in \mathbf{G}$ and arbitrary $\theta_g \in \Theta$.
Since $\mathbf{G}$ does $\Theta$-PR, we obtain $f=\theta\cdot 0 = 0$. Since $f$ was arbitrary, it follows that $\mathbf{G}$ is complete.

    Now, let $\mathbf{G}$ be complete and suppose that $\Theta=\{\theta\}$ for some $\theta \in \T$. Fix any $f,h\in H$ and suppose that for every $g\in \mathbf{G}$ we have $$\langle f,g\rangle = \theta_g \langle h,g\rangle$$ with $\theta_g \in \Theta$. By assumption on $\Theta$ we have $\theta_g = \theta$ for all $g \in \mathbf{G}$. Hence, $\langle f-\theta h,g\rangle = 0$ for all $g \in \mathbf{G}$ and completeness yields $f=\theta h$.
\end{proof}

Using the statement above we can easily derive the following lemma, that will be applied throughout this article.

\begin{lemma}\label{lemma:theta_pr_implies_theta'_pr}
Let $\Theta, \Theta' \subseteq \T$ such that $\Theta' \subseteq \Theta$. If $\mathbf{G} \subseteq H$ does $\Theta$-PR then $\mathbf{G}$ does $\Theta'$-PR.
\end{lemma}
\begin{proof}
    Suppose $\mathbf{G}$ does $\Theta$-PR. In order to show that $\mathbf{G}$ does $\Theta'$-PR, we fix $f,h\in H$ and suppose that for every $g\in\mathbf{G}$ there exists $\theta_g\in \Theta'$ such that
\begin{equation}\label{eq:theta'_implies_theta}
    \langle f,g\rangle =\theta_g \langle h,g\rangle.     
\end{equation}
Since $\Theta'\subseteq \Theta$, and since $\mathbf{G}$ does $\Theta$-PR, it follows that $f=\theta h$ for some $\theta \in \Theta$. It remains to show that $\theta \in \Theta'$. Suppose $f=\theta h$ for some $\theta \in \Theta \setminus \Theta'$,. Using equation \eqref{eq:theta'_implies_theta} in combination with the property $\theta - \theta_g \neq 0$ for all $g \in \mathbf{G}$, we obtain that $h \in \mathbf{G}^\perp$. By Lemma \ref{lma:completeness}, since $\mathbf{G}$ is complete we get $h = 0$ and thus also $f=0$, which shows that $f=\theta h$ for any $\theta \in \Theta'$, giving a contradiction.
\end{proof}

Another useful observation that we will use throughout the article is the following.

\begin{lemma}\label{lma:failure_witnessing}
    Let $\Theta\subseteq \T$ and let $\mathbf{G}
    \subseteq H$. Then $\mathbf{G}$ fails $\Theta$-PR if and only there exist linearly independent $f,h \in H$ witnessing this failure.
\end{lemma}
\begin{proof}
    It suffices to show the forward implication, as the reverse implication is immediate. We consider two cases.
    
    \textbf{Case 1}: $\mathbf{G}$ is complete. Since $\mathbf{G}$ fails $\Theta$-PR we can find $f,h\in H$ such that for every $g\in \mathbf{G}$ there exists $\theta_g\in \Theta$ with
    $$\langle f,g\rangle = \theta_g\langle h,g\rangle,$$
    and yet $f\neq \theta h$ for any $\theta \in \Theta$. Observe that in particular this gives $f,h\neq 0$. Suppose that $f=ch$ for some $c\in \C\setminus \Theta$. Then for every $g \in \mathbf{G}$ 
    $$c\langle h,g\rangle=\langle f,g\rangle = \theta_g\langle h,g\rangle.$$
    Since $h \neq 0$ and $\mathbf{G}$ is complete, we can find $g \in \mathbf{G}$ such that $\langle h,g\rangle\neq 0$. This gives $c=\theta_g \in \Theta$, a contradiction.

    \textbf{Case 2:} $\mathbf{G}$ is not complete. If $\mathbf{G}=\{0\}$, then any linearly independent $f,h\in H$ witness this failure. Suppose $\mathbf{G}\neq \{0\}$. We pick any non-zero $f\in \mathbf{G}^\perp$, $h\in \mathbf{G}$ and any $\theta \in \Theta$. We first observe that $x=f+\theta h$ and $h$ are linearly independent. Indeed, if $f+\theta h = ch$ for some $c\in \C$, then $f+(\theta -c)h=0$. If $c=\theta$ then $f=0$, giving a contradiction. If $c\neq \theta$, then $f = (c-\theta)h$, hence $f\in \mathbf{G}^\perp \cap \mathbf{G}$, thus $f=0$, also resulting in a contradiction. Finally, observe that for every $g\in \mathbf{G}$ we have
    $$\langle x,g \rangle = \langle f+\theta h,g\rangle = \theta \langle h,g\rangle,$$
    which shows that the linearly independent vectors $x$ and $h$ witness the failure of $\Theta$-PR.
\end{proof}

We also make use of the following elementary fact.

\begin{lemma}\label{lemma:invariantTransform}
Let $H_1,H_2$ be Hilbert spaces, let $\mathbf{G} \subseteq H_1$, and let $\Theta \subseteq \T$.
    \begin{enumerate}
        \item If $T : H_1 \to H_2$ is an invertible bounded linear operator, then $\mathbf{G}$ does $\Theta$-PR in $H_1$ if and only if $T(\mathbf{G}) = \{ Tg : g \in \mathbf{G} \}$ does $\Theta$-PR in $H_2$.
        \item If $\{ c_g \}_{g \in \mathbf{G}} \subseteq \C \setminus \{ 0 \}$, then $\mathbf{G}$ does $\Theta$-PR in $H_1$ if and only if $\{ c_gg : g \in \mathbf{G} \}$ does $\Theta$-PR in $H_1$.
    \end{enumerate}
\end{lemma}
\begin{proof}
    Part (2) of the statement follows directly from the definition of a system doing $\Theta$-PR. To prove part (1), suppose that $\mathbf{G}$ does $\Theta$-PR in $H_1$ and let $f,h \in H_2$ such that
    $$
    \langle f,Tg \rangle_{H_2} = \theta_g \langle h,Tg \rangle_{H_2}
    $$
    for all $g \in \mathbf{G}$ and some $\theta_g \in \Theta$. The latter is equivalent to the condition that the adjoint $T^* : H_2 \to H_1$ of $T$ satisfies
    $$
    \langle T^*f,g \rangle_{H_1} = \theta_g \langle T^*h,g \rangle_{H_1}
    $$
    for all $g \in \mathbf{G}$, which implies that $T^* f = \theta T^* h$ for some $\theta \in \Theta$. Since $T^*$ is linear and invertible, we have $f= \theta h$, which implies that $T(\mathbf{G})$ does $\Theta$-PR in $H_2$. The reverse implication follows analogously.
\end{proof}

Lemma \ref{lma:completeness} shows that the completeness of a system is sufficient for doing $\Theta$-PR for singletons. If $\Theta$ consists of more than one element, then $\Theta$-PR can be achieved by replacing completeness with the notion of overcompleteness. Recall that a countable system $\mathbf{G} = \{ g_j \}_{j\in \N} \subseteq H$ is said to be overcomplete (or hypercomplete or densely closed) if every subsequence $\{ g_{j_k} \}_{k\in\N}$ of $\mathbf{G}$ is complete \cite{singer1981bases}. Here we use the classical notion of overcompleteness commonly found in the Banach space literature; alternative notions appear in the theory of frames \cite{christensen2003introduction,heil2011basis}. A concrete example of an overcomplete system is provided by exponential systems $E(\Lambda)$ with $\Lambda$ satisfying a certain topological property, as described below. Recall that for $\Lambda \subseteq \C$ the exponential system $E(\Lambda)$ consists of all exponentials $e_\lambda(x) = e^{2\pi i \lambda x}$ with $\lambda \in \Lambda$. It follows from \cite[Theorem 4.3]{chalendar2006overcompleteness}, that for $a>0$, the exponential system $E(\Lambda)$ is overcomplete in $L^2[0,a]$ if and only if $\Lambda$ has an accumulation point.
Our next result provides a comparison between overcompleteness and $\Theta$-PR.

\begin{lemma}\label{prop:overcompleteness}
    If $\mathbf{G} \subseteq H$ is overcomplete then $\mathbf{G}$ does $\Theta$-PR for every finite $\Theta\subseteq \T$.
\end{lemma}
\begin{proof}
    Suppose that $\mathbf{G} = \{ g_n \}_{n \in \N}$ is overcomplete and let $\Theta$ be finite with $\Theta = \{ \theta_1, \dots, \theta_N \}$. To show that $\mathbf{G}$ does $\Theta$-PR let $f,h \in H$ be such that
    $$
    \langle f,g_n \rangle = \theta_{g_n} \langle h,g_n \rangle
    $$
    with $\theta_{g_n} \in \Theta$. Define index sets $I_j$ via
    $$
    I_j = \{ n \in \N : \langle f,g_n \rangle = \theta_j \langle h,g_n \rangle \}.
    $$
    Then $\N = I_1 \cup \dots \cup I_N$ and at least one of the index sets $I_j$ is infinite. Denote this set by $I_\ell$. Since $\mathbf{G}$ is overcomplete it follows that $\{ g_n \}_{n \in I_\ell}$ is complete and satisfies
    $$
    \langle f-\theta_\ell h, g_n \rangle =0
    $$
    for all $n \in I_\ell$. Hence $f-\theta_\ell h = 0$.
\end{proof}

The reverse implication of the latter statement is in general not true, as the following example shows.

\begin{example}
    Consider the exponential system $\mathbf{G} = E(-i\N) = \{ e^{2\pi n x}\}_{n \in \N}$ in $H=L^2[0,1]$.  According to \cite{chalendar2006overcompleteness}, $\mathbf{G}$ is not overcomplete in $L^2[0,1]$. On the other hand, $\mathbf{G}$ does $\Theta$-PR for any finite $\Theta$: let $\Theta = \{ \gamma_1, \dots, \gamma_N \}$ and $f,h \in H$ with
    $$
    \langle f,e^{2\pi nx} \rangle = \theta_n \langle f,e^{2\pi nx} \rangle
    $$
    for some $\theta_n \in \Theta$. Let
    $$
    I_j = \{ n \in \N : \langle f,e^{2\pi n x } \rangle = \gamma_j \langle h, e^{2\pi n x } \rangle \}.
    $$
    Since
    $$
    \infty = \sum_{n=1}^\infty \frac{1}{n} = \sum_{j=1}^N \sum_{n \in I_j} \frac{1}{n}
    $$
    there exists $\ell$ such that $\sum_{n \in I_\ell} \frac{1}{n} = \infty$. By a change of variables and the Müntz-Szasz Theorem \cite{Young} it follows that $\{ e^{2\pi n x}\}_{n \in I_\ell}$ is complete in $L^2[0,1]$. Since
    $$
    \langle f-\theta_\ell h, e^{2\pi n x} \rangle =0
    $$
    for all $n \in I_\ell$, from completeness of $\{e^{2\pi nx}\}_{n\in I_\ell}$ we get $f=\theta_\ell h$. Hence, $\{ e^{2\pi n x}\}_{n \in \N}$ does $\Theta$-PR. Since $\Theta$ was an arbitrary finite set, the statement follows.
\end{example}

Finally, we point out that overcompleteness is in general incomparable to PR.

\begin{example}
    (1) Consider the Hilbert space $H = L^2[0,1]$ and the set
    $$
    S = \left \{ \frac{1}{n} : n \in \N \right \}
    $$
    Moreover, consider the full exponential system $E(\R)$. It is well-known that $E(\R)$ fails PR in $L^2[0,1]$. Since $E(S)\subseteq E(\R)$, we get that $E(S)$ also fails PR in $L^2[0,1]$. On the other hand, it follows from \cite{chalendar2006overcompleteness} that $E(S)$ is overcomplete in $L^2[0,1]$. Hence, overcompleteness does not imply PR.

    (2) Consider the Hilbert space $H=L^2(\R)$ and the set
    $$
    S = \left \{ (\pm \sqrt{n},\pm \sqrt{k}) : (n,k) \in \N^2 \right \} \subseteq \R^2.
    $$
    In \cite{grohs2025phaselesssquareroot} it was shown that there exists $g \in \lt$ such that the Gabor system
    $$
    \mathbf{G} = \{ e^{2\pi i \omega x} g(x-t): (t,\omega) \in S \}
    $$
    does PR in $L^2(\R)$. Notice that the system of translates
    $$
    \mathbf{G}' = \{g(x-n) : n \in \Z \}
    $$
    forms an incomplete subset of $\mathbf{G}$. This follows from the well-known fact that a system of integer translates of a function $g \in \lt$ is never complete in $\lt$. Considering $\mathbf{G}$ as a sequence in $\lt$ and $\mathbf{G}'$ as a subsequence of $\mathbf{G}$ shows that $\mathbf{G}$ is not overcomplete but it does PR. 
\end{example}

\subsection{Complement property}\label{sec:complement_property}

Lemma \ref{lma:completeness} and Lemma \ref{prop:overcompleteness} show that completeness and overcompleteness provide sufficient conditions for $\Theta$-PR for singletons and finite sets, respectively. However, neither of these conditions yields a characterization of $\Theta$-PR. Thus, to obtain a genuine characterization, one must replace completeness by a stronger condition - yet one that remains strictly weaker than overcompleteness. We will obtain such a characterization in Section \ref{sec:characteriztion} for every countable set $\Theta \subseteq \T$. In the present section we focus on the special case when $\Theta$ consists of two elements.

In the case of real Hilbert spaces, a characterization for the sign retrieval problem has been established in terms of the so-called complement property. This property may be viewed as a strengthened form of completeness: whereas completeness requires that the closure of $\lspan(\mathbf{G})$ is all of $H$, the complement property imposes a more robust spanning condition on every two-element partition of $\mathbf{G}$.

\begin{definition}
    We say that $\mathbf{G} \subseteq H$ has the \emph{complement property} if for every $\mathbf{S} \subseteq \mathbf{G}$ it holds that $\mathbf{S}$ is complete in $H$ or $\mathbf{G} \setminus \mathbf{S}$ is complete in $H$. 
\end{definition}

The following proposition establishes a fundamental link between phase retrieval and the complement property \cite{cahill2016phase}.

\begin{proposition}\label{prop:cp2}
    If $\mathbf{G} \subseteq H$ does PR, then $\mathbf{G}$ has the complement property. If $H$ is a real Hilbert space and $\mathbf{G} \subseteq H$ has the complement property, then $\mathbf{G}$ does (real) PR.
\end{proposition}

According to the latter result, the complement property characterizes phase retrieval only in the setting of real Hilbert spaces. Proposition \ref{prop:cp} shows that, in fact, the complement property always characterizes $\Theta$-PR whenever $\Theta$ consists of exactly two elements. In this broader framework, the distinction between real and complex phase retrieval disappears, and Proposition \ref{prop:cp2} becomes an immediate corollary.

\begin{proof}[Proof of Proposition \ref{prop:cp}]
    \textbf{Necessity}. Assume that $\mathbf{G}$ does $\{\theta_1,\theta_2\}$-PR and $\mathbf{G}$ does not have the complement property. Then there exists $\mathbf{S} \subseteq \mathbf{G}$ such that both $\mathbf{S}$ and $\mathbf{G} \setminus \mathbf{S}$ are not complete. Therefore, we can find non-zero $f,h \in H$ such that $\langle f,g\rangle =0$ for all $g\in \mathbf{S}$ and $\langle h,g\rangle =0$ for all $g\in \mathbf{G\setminus S}$. Hence,
    $$
    \langle \theta_1 f+\theta_2 h, g\rangle =\begin{cases} 
      \theta_1\langle f,g\rangle, & g\in \mathbf{G\setminus S} \\
      \theta_2\langle h,g\rangle, & g\in \mathbf{S} 
   \end{cases}.
    $$
    The latter implies that, for every $g \in \mathbf{G}$ there exists $\theta_g \in\{\theta_1,\theta_2\}$ such that
    \begin{equation}
        \langle \theta_1 f+\theta_2 h, g\rangle  = \theta_g \langle f+h, g\rangle.
    \end{equation}
    Since by assumption $\mathbf{G}$ does $\{\theta_1,\theta_2\}$-PR, it follows that there exists $\theta\in \{\theta_1,\theta_2\}$ such that  $\theta_1 f+\theta_2 h = \theta(f+h)$, or equivalently, $(\theta_1-\theta)f + (\theta_2-\theta)h=0$. If $\theta=\theta_1$, then from $h\neq 0$ we get $\theta_2=\theta$, which contradicts the assumption that $\theta_1 \neq \theta_2$. 
    The case when $\theta = \theta_2$ leads to an analogous contradiction.

    \textbf{Sufficiency}. Assume that $\mathbf{G}$ has the complement property but fails $\{\theta_1,\theta_2\}$-PR. Then we can find $f,h\in H$, with $f\neq \theta_1h$ and $f\neq \theta_2h$ such that for every $g\in \mathbf{G}$ there exists $\theta_g \in \{\theta_1,\theta_2\}$ so that
    $$
    \langle f,g\rangle = \theta_g\langle h,g\rangle
    $$
    Define $\mathbf{S} \coloneqq \{g\in \mathbf{G}: \langle f,g\rangle=\theta_1\langle h,g\rangle\}$. Using the complement property, at least one of $\mathbf{S}$ and $\mathbf{G\setminus S}$ is complete. Suppose the former. Then since $\langle f-\theta_1 h, g\rangle =0$ for all $g\in \mathbf{S}$ we obtain $f-\theta_1 h=0$, giving a contradiction. The case where $\mathbf{G\setminus S}$ is complete gives a similar contradiction.
\end{proof}

Notice, that the latter proposition implies that if $\Theta \subseteq \T$ and $\Theta' \subseteq \T$ satisfy $|\Theta| = |\Theta'|=2$, then $\mathbf{G} \subseteq H$ does $\Theta$-PR if and only if it does $\Theta'$-PR. This observation motivates the following definition.

\begin{definition}\label{def:2PR}
    We say that $\mathbf{G}\subseteq H$ does $2$-PR if it does $\Theta$-PR for any (hence all) $\Theta\subseteq \T$ with $|\Theta|=2$.
\end{definition}

Next, we consider the question how real phase retrieval fits into the framework of 2-PR when one passes to the complexification of a Hilbert space. Recall that given a real Hilbert space $H$ we define its complexification as a complex Hilbert space $H_\C$ given by $H_\C = H\oplus iH$, with addition 
$$(f_1 + if_2) + (h_1 + ih_2) := (f_1+h_1) + i(f_2+h_2)$$
and multiplication by complex scalars $\lambda = \lambda_1+i\lambda_2$
$$(\lambda_1+i\lambda_2)(f_1+if_2):=(\lambda_1f_1 - \lambda_2f_2) + i(\lambda_2f_1 + \lambda_1f_2).$$
The inner product on $H_\C$ is defined in the natural way 
$$\langle f_1+if_2,h_1+ih_2\rangle:=\langle f_1,h_1\rangle - \langle f_2,h_2\rangle + i(\langle f_2,h_1\rangle + \langle f_1,h_2\rangle).$$

Although real and complex phase retrieval behave differently in general, the case of two phases is special: the two-phase ambiguity already captures exactly the sign ambiguity inherent in real phase retrieval. This connection is made precise in the next statement.

\begin{corollary}
    Let $H$ be a real Hilbert space and let $\mathbf{G} \subseteq H$. Then $\mathbf{G}$ does (real) PR in $H$ if and only if $\mathbf{G}$ does $2$-PR in $H_\C$.
\end{corollary}
\begin{proof}
    Since real PR and $2$-PR are equivalent to the (real and complex) complement property, it's enough to show that $\mathbf{G}$ has the complement property in $H$ if and only if it has the complement property in $H_\C$. We fix $\mathbf{S}\subseteq \mathbf{G}$. In both of the implications we additionally assume $\mathbf{S}$ is complete, as the second case of $\mathbf{G\setminus S}$ being complete is proved the same way.

    \textbf{Necessity}. Assume $\mathbf{S}$ is complete in $H$. To show that $\mathbf{S}$ is complete in $H_\C$, fix any vector $f\in H_\C$ such that for every $g\in \mathbf{S}$ we have $\langle f,g\rangle=0.$ We write $f=f_1+if_2$ where $f_1,f_2\in H$. Then for every $g\in \mathbf{S}$ we have
    $$0 = \langle f,g\rangle = \langle f_1, g\rangle + i\langle f_2,g\rangle.$$
    Since $\langle f_j,g\rangle \in \R$ for $j=1,2$, this forces $\langle f_j,g\rangle=0$ for $j=1,2$. Due to completeness of $\mathbf{S}$ in $H$ we get $f_1,f_2=0$. Thus $f=0$, which gives completeness of $\mathbf{S}$ in $H_\C$.

    \textbf{Sufficiency}. Assume $\mathbf{S}$ is complete in $H_\C$. Fix any $f\in H$ such that $\langle f,g\rangle=0$ for every $g\in \mathbf{S}$. Since $\mathbf{S}$ is complete in $H_\C$ and since $f$ can be identified as an element of $H_\C$, we get that $f=0$, which gives completeness of $\mathbf{S}$ in $H$.
\end{proof}

\section{Characterization of $\Theta$-PR}\label{sec:characteriztion}

\subsection{Characterization for countable $\Theta$}

In this section we provide a concrete characterization of all systems in a Hilbert space that do $\Theta$-PR for a countable set $\Theta\subseteq \T$. When $|\Theta| = 2$, such a characterization was obtained in Section \ref{sec:complement_property} in terms of the complement property. Theorem \ref{thm:m1} provides a substantial generalization of this criterion to arbitrary countable phase sets $\Theta$. 

\begin{proof}[Proof of Theorem \ref{thm:m1}]
\textbf{Necessity.}
Assume that $\mathbf{G}$ fails $\Theta$-PR. We have to construct a cover $\{ \mathbf{G}_j \}_{j \in \N}$ and a sequence $\{ x_j \}_{j \in \N} \subseteq H$ satisfying properties (1), (2) and (3). By Lemma \ref{lma:failure_witnessing}, we can pick linearly independent $f,h \in H$ such that for every $g \in \mathbf{G}$ there exists $\theta_g \in \Theta$ with
$$
\langle f,g \rangle = \theta_g \langle h,g \rangle.
$$
For every $j \in \N$ define $\mathbf{G}_j \subseteq \mathbf{G}$ via
$$
\mathbf{G}_j \coloneqq \{ g \in \mathbf{G} : \langle f,g \rangle = \theta_j \langle h,g \rangle \}.
$$
Since for every $g \in \mathbf{G}$ there exists $j \in \N$ such that $\langle f,g \rangle = \theta_j \langle h,g \rangle$, it follows that $\{ \mathbf{G}_j \}_{j \in \N}$ is a cover of $\mathbf{G}$. Further, define $x_j \in H$ by
$$
x_j \coloneqq f-\theta_j h.
$$
Since $f,h$ are linearly independent we have $x_j \neq 0$ for all $j \in \N$. Moreover, for every $g \in \mathbf{G}_j$ we have
$$
\langle x_j, g \rangle =0
$$
which implies that $x_j \in \mathbf{G}_j^\perp$. Next, we show that $x_1,x_2$ are linearly independent: suppose by contradiction that $x_1,x_2$ are linearly dependent. Then there exists $\alpha \in \C$ such that $x_1 = \alpha x_2$ which is equivalent to
$$
f-\theta_1 h = \alpha ( f-\theta_2 h ).
$$
Rearranging shows that
$$
(1-\alpha) f + (-\theta_1 + \alpha \theta_2) h = 0.
$$
Since $f,h$ are linearly independent we have
$$
1-\alpha = -\theta_1 + \alpha \theta_2 = 0.
$$
The latter implies that $\theta_1 = \theta_2$ which contradicts the assumption $\theta_1 \neq \theta_2$.

It remains to show property (3). To do so, we observe that for every $z \in \C$ we have the vector identity
\begin{equation}\label{eq:vi}
    (\theta_2 - \theta_1)(f-zh) = (z-\theta_1)(f-\theta_2 h) - (z-\theta_2)(f-\theta_1h).
\end{equation}
Setting $z = \theta_j$, dividing by $\theta_2 - \theta_1 \neq 0$, and using the definition of $x_j$, it follows that
$$
x_j = \frac{\theta_j - \theta_1}{\theta_2 - \theta_1} x_2 - \frac{\theta_j - \theta_2}{\theta_2 - \theta_1} x_1.
$$

\textbf{Sufficiency.}
Assume that $\{ \mathbf{G}_j \}_{j \in \N}$ and $\{ x_j \}_{j \in \N}$ exist and satisfy (1), (2) and (3). We have to show that $\mathbf{G}$ fails $\Theta$-PR. To do so, let $f$ and $h$ be the unique solution to the system given by
$$
f-\theta_1 h = x_1, \quad f-\theta_2 h = x_2.
$$
Note that such a unique solution exists as the matrix $\begin{pmatrix}
    1 & -\theta_1 \\
    1 & -\theta_2
\end{pmatrix}$ is invertible. Moreover, linear independence of $x_1$ and $x_2$ gives linear independence of $f$ and $h$. Hence to show that $\mathbf{G}$ fails $\Theta$-PR on $f$ and $h$, it's enough to show that for every $g\in \mathbf{G}$ we have $\langle f,g\rangle = \theta_g\langle h,g\rangle$ for some $\theta \in \Theta$.

First, from the definition of $x_1$ and $x_2$ together with property (1), if $g\in \mathbf{G}_j$ for $j=1,2$ then $\langle x_j,g\rangle=0$, which gives $\langle f,g\rangle = \theta_j\langle h,g\rangle$. Next, fix $j\geq 3$. Using the vector identity \eqref{eq:vi} with $z=\theta_j$ and property (3) we get 
$$
f-\theta_j h = \frac{\theta_j - \theta_1}{\theta_2 - \theta_1} x_2 - \frac{\theta_j - \theta_2}{\theta_2 - \theta_1} x_1 = x_j.
$$
Since $x_j\in \mathbf{G}^\perp_j$ by assumption, we have $\langle f-\theta_j h,g\rangle = \langle x_j,g\rangle=0$ for all $g\in \mathbf{G}_j$, or equivalently $\langle f,g\rangle = \theta_j\langle h,g\rangle$. 
\end{proof}

\subsection{Cyclic $\Theta$}

If the set $\Theta$ carries an algebraic structure in the sense that it is generated by a single phase $\omega \in \Theta$, then Theorem \ref{thm:m1} admits a particularly simple reformulation. In this case, the criterion for failure of $\Theta$-PR can be expressed in terms of a second-order recurrence relation.

\begin{proof}[Proof of Proposition \ref{prop:recurrence}] 
    Since $\omega \in \T\setminus\{1\}$ it holds that $\omega \neq \omega^2$. According to Theorem \ref{thm:m1}, it suffices to show that
    \begin{equation}\label{eq:rec_rel}
        x_ j = \frac{\omega^j-\omega}{\omega^2-\omega}x_2 - \frac{\omega^j-\omega^2}{\omega^2-\omega}x_1, \quad j \geq 3,
    \end{equation}
    is equivalent to the recurrence relation
    $$
    x_j = (1+\omega)x_{j-1} - \omega x_{j-2}, \quad j \geq 3.
    $$
    To do so, observe that the relation \eqref{eq:rec_rel} is equivalent to
\[
x_j = \frac{\omega^{j-1}-1}{\omega-1}\,x_2 - \frac{\omega^{j-1}-\omega}{\omega-1}\,x_1.
\]
For $j \geq 3$, let $A_j=\sum_{k=0}^{j-2} \omega^k$ and $B_j=\omega \sum_{k=0}^{j-3} \omega^k$.
Using the identities
$$
\frac{\omega^{j-1}-1}{\omega-1} = \sum_{k=0}^{j-2} \omega^k,
\quad
\frac{\omega^{j-1}-\omega}{\omega-1} = \sum_{k=1}^{j-2} \omega^k,
$$
we can rewrite $x_j$ as
$$
x_j = A_j x_2
-
B_j x_1.
$$
A direct calculation shows that
$$
A_j = (1+\omega)A_{j-1} - \omega A_{j-2}, \quad j \geq 3.
$$
Indeed, since $A_j = A_{j-1} + \omega^{j-2}$ and
$A_{j-1} = A_{j-2} + \omega^{j-3}$, we have
$$
(1+\omega)A_{j-1} - \omega A_{j-2}
= (1+\omega)A_{j-1} - \omega(A_{j-1} - \omega^{j-3})
= A_j.
$$
An analogous calculation implies that
$$
B_j = (1+\omega)B_{j-1} - \omega B_{j-2}, \quad j \geq 3.
$$
Therefore, for every $j \geq 3$, we have
\begin{align*}
x_j
&= A_j x_2 - B_j x_1 \\
&= (1+\omega)(A_{j-1} x_2 + B_{j-1} x_1)
   - \omega(A_{j-2} x_2 + B_{j-2} x_1) \\
&= (1+\omega)x_{j-1} - \omega x_{j-2}.
\end{align*}
\end{proof}

If $\Theta$ consists of three elements, from Theorem \ref{thm:m1} we can also deduce that the ability of a system to do $\Theta$-PR is independent of the particular elements in $\Theta$ and depends only on the cardinality of the set. In this special case we get the following characterization, which will be used in Section \ref{sec:failure} to study failure of $\Theta$-PR for systems in $\C^d$

\begin{corollary}\label{cor:3_pr_characterization}
    Let $\Theta = \{\theta_1,\theta_2,\theta_3\}\subseteq \T$ be with pairwise distinct elements. Then $\mathbf{G} \subseteq H$ fails $\Theta$-PR if and only if there exist $\mathbf{G}_1, \mathbf{G}_2, \mathbf{G}_3\subseteq \mathbf{G}$ with $\mathbf{G}_1 \cup \mathbf{G}_2 \cup \mathbf{G}_3 = \mathbf{G}$ and $x_1,x_3,x_3 \in H$ such that 
    \begin{enumerate}
        \item $0\neq x_j \in \mathbf{G}_j^\perp$ for $j=1,2,3$,
        \item $x_1$ and $x_2$ are linearly independent,
        \item $x_3 \in \lspan\{x_1,x_2\}$.
    \end{enumerate}
\end{corollary}
\begin{proof}
    The forward implication is an immediate consequence of Theorem \ref{thm:m1} with $\Theta$ consisting of three elements. Hence, it suffices to show the reverse implication.

    To that end, suppose there exist a cover $\{ \mathbf{G}_1,\mathbf{G}_2,\mathbf{G}_3 \}$ of $\mathbf{G}$ and $x_1,x_2,x_3 \in H$ satisfying the conditions (1), (2) and (3). 

    If $x_3 = ax_1$ for some $a \neq 0$, then $\mathbf{G}$ fails the complement property since $\mathbf{S}=\mathbf{G}_2$ and $\mathbf{G} \setminus \mathbf{S}  = \mathbf{G}_1 \cup \mathbf{G}_3$ are both incomplete. Thus, $\mathbf{G}$ fails $2$-PR by Proposition \ref{prop:cp} and therefore it also fails $\Theta$-PR from Lemma \ref{lemma:theta_pr_implies_theta'_pr}.

    An analogous argument as before implies that if $x_3=ax_2$ for some $a \neq 0$, then $\mathbf{G}$ fails $\Theta$-PR.
    
    It remains to consider the case when $x_3 =ax_1+bx_2$ for some $a,b\neq 0$. In this case, define auxillary vectors $x_1',x_2',x_3'$ via
$$
x_1' = -a\frac{\theta_2-\theta_1}{\theta_3-\theta_1}x_1, \quad x_2'=b\frac{\theta_2-\theta_1}{\theta_3-\theta_2}x_2, \quad  x_3'=x_3.
$$
Then $0 \neq x_j' \in \mathbf{G}_j^\perp$ for every $j \in \{ 1,2,3 \}$. Moreover $x_1'$ and $x_2'$ are linearly independent. Finally
    \begin{align*}
        \frac{\theta_3-\theta_1}{\theta_2-\theta_1}x_2' - \frac{\theta_3-\theta_2}{\theta_2-\theta_1}x_1' = ax_1 + bx_2 = x_3'.
    \end{align*}
    Consequently, Theorem \ref{thm:m1} shows that $\mathbf{G}$ fails $\Theta$-PR.
\end{proof}

Since the latter corollary shows that the property of a system doing $\Theta$-PR for a three-element set $\Theta$ is independent of the choice of elements in $\Theta$, this suggests the following definition, which may be viewed as a natural extension of Definition \ref{def:2PR}.

\begin{definition}
    We say $\mathbf{G}\subseteq H$ does $3$-PR if it does $\Theta$-PR for any (hence all) $\Theta\subseteq \T$ with $|\Theta|=3$.
\end{definition}

A natural question to ask is if this trend continues for all finite $\Theta\subseteq \T$. In Section \ref{sec:moebius} we will show that given $\Theta,\Theta'\subseteq\T$ with $|\Theta|=|\Theta'|=4$, the properties of doing $\Theta$-PR and $\Theta'$-PR are equivalent, provided that $\Theta$ and $\Theta'$ have the same cross ratios. Nevertheless, in Section \ref{sec:c2_characterization} we will observe that without extra assumptions on $\Theta$ and $\Theta'$, such an equivalence is in general false.

\subsection{Exponential systems}

This subsection is concerned with the proof of Theorem \ref{thm:lattice_density_characterization} and Theorem \ref{thm:ER_disconnected}. In order to prove these statements we start with some preliminary definitions and observations.

We define the Fourier transform $\ft f$ of a function $f \in L^1(\R) \cap \lt$ by
$$
\ft f(s) = \int_\R f(t) e^{-2\pi i s t} \, dt.
$$
The Fourier transform $\ft$ extends from $L^1(\R) \cap \lt$ to a unitary operator on $\lt$ in the usual way.
For $a >0$, we let $PW_a$ be the Paley-Wiener spaces of functions $f \in \lt$ such that the support of the Fourier transform of $f$ is contained in $[-\frac{a}{2},\frac{a}{2}]$,
$$
PW_a \colonequals \left \{ f \in \lt : \supp(f) \subseteq \left [-\frac{a}{2},\frac{a}{2} \right ] \right \}.
$$
We also define the abbreviation $PW \coloneqq PW_1$. The convolution of $f,g$ is given by
$$
(f * g)(x) = \int_\R f(x-y) g(y) \, dy,
$$
whenever this expression is well-defined. Recall that the support of the convolution satisfies the relation
$$
\supp(f * g) \subseteq \mathrm{cl} ( \supp(f) + \supp(g) )
$$
where $\mathrm{cl}(X)$ denotes the closure of a set $X$.

In order to prove Theorem \ref{thm:lattice_density_characterization}, we start with a convenient reformulation of Proposition \ref{prop:recurrence} which says that the study of $\Theta$-PR for an exponential system reduces to the study of zero-sets of functions in the Paley-Wiener space.

\begin{corollary}\label{cor:PW_recurrence}
Let $\Theta = \{ e^{2\pi i k/n} : k=0,\dots,n-1 \}$, let $\Lambda \subseteq \R$ and let $a>0$. Then $E(\Lambda)$ fails $\Theta$-PR in $L^2[0,a]$ if and only if there exist a cover $\{\Lambda_j\}_{j \in \N}$ of $\Lambda$ and functions $\{ x_j \}_{j \in \N} \subseteq PW_a \setminus \{ 0 \}$ satisfying
\begin{enumerate}
        \item $x_j$ vanishes on $\Lambda_j$,
        \item $x_1$ and $x_2$ are linearly independent,
        \item for every $j \geq 3$, it holds that $x_j = (1+e^{\frac{2\pi i }{n}})x_{j-1} - e^{\frac{2\pi i }{n}} x_{j-2}$.
    \end{enumerate}
\end{corollary}
\begin{proof}
    It suffices to consider the case $a=1$ since the general case follows from scaling.
    
    We start by observing that $E(\Lambda)$ does $\Theta$-PR in $L^2[0,1]$ if and only if $E(\Lambda)$ does $\Theta$-PR in $L^2[-\frac{1}{2},\frac{1}{2}]$. By Proposition \ref{prop:recurrence} it suffices to show that the existence of a cover $\{\Lambda_j\}_{j \in \N}$ of $\Lambda$ and functions $\{ x_j \}_{j \in \N} \subseteq PW \setminus \{ 0 \}$ satisfying (1), (2) and (3) is equivalent to the existence of a cover $\{\mathbf{G}_j\}_{j \in \N}$ of $E(\Lambda)$ and a sequence $\{ y_j \}_{j \in \N} \subseteq L^2[-\frac{1}{2},\frac{1}{2}]$ satisfying
\begin{enumerate}
        \item $0 \neq y_j \in \mathbf{G}_j^\perp$ for every $j\in \N$,
        \item $y_1$ and $y_2$ are linearly independent,
        \item for every $j \geq 3$, the vector $y_j$ is given by the second-order recurrence relation $y_j = (1+\omega)y_{j-1} - \omega y_{j-2}$ with $\omega = e^{\frac{2\pi i }{n}}$.
    \end{enumerate}

    To do so, we observe that the functions $x_j \in PW$ satisfy for every $\lambda \in \Lambda$ the relation
    \begin{equation}\label{eq:reformulation}
        x_j(\lambda) = \ft \ift x_j(\lambda) = \langle \ift x_j, e_\lambda \rangle.
    \end{equation}
    Since $\ft$ is unitary we have that $x_j \neq 0$ if and only if $\ift x_j \neq 0$. Defining $y_j \coloneqq \ift x_j$ and $\mathbf{G}_j \coloneqq E(\Lambda_j)$ implies that $x_j$ vanishes on $\Lambda_j$ if and only if $y_j \in \mathbf{G}_j^\perp$. Moreover, $\supp(y_j) \subseteq [-\frac{1}{2},\frac{1}{2}]$ and hence every $y_j$ can be identified with an element in $L^2[-\frac{1}{2},\frac{1}{2}]$. The fact that properties (2) and (3) in the language of covers of $\Lambda$ are equivalent to properties (2) and (3) in the language of covers of $E(\Lambda)$ follow directly from the definition of $y_j$ and $\mathbf{G}_j$.
\end{proof}

We are ready to derive Theorem \ref{thm:lattice_density_characterization}

\begin{proof}[Proof of Theorem \ref{thm:lattice_density_characterization}]
Without loss of generality, we can assume that $a=1$. The general result follows by scaling. Moreover, we can assume that $n \geq 2$, as the case $n=1$ is a consequence of the classical Fourier uniqueness theorem.

\textbf{Sufficiency.}
For $\lambda \in \Lambda$ define as usual $e_\lambda(x) = e^{2\pi i \lambda x}$. Let $f,h \in L^2[0,1]$ and suppose that for every $\lambda \in \Lambda$ there exists $\theta_\lambda \in \Theta$ such that
\begin{equation}\label{eq:eeeee}
    \langle f,e_\lambda \rangle = \theta_\lambda \langle h,e_\lambda \rangle.
\end{equation}
Using the definition of the Fourier transform it follows that for $\lambda \in \Lambda$, the identity \eqref{eq:eeeee} is equivalent to the property that
$$
\ft(f-\theta_\lambda h)(\lambda) =0.
$$
Now let $\omega_j=e^{2\pi i j/n}$ and define for $j \in \{ 0, \dots, n-1 \}$ functions $p_j \in L^2[0,1]$ via
\begin{equation}\label{eq:prod_analytic}
        p_j \coloneqq  f - \omega_j h.
    \end{equation}
Consider each $p_j$ as a function in $L^2(\R)$ with support in $[0,1]$. By our assumption on $f,h$ we have that
$$
\prod_{j=0}^{n-1} \ft(p_j)(\lambda) = 0, \quad \lambda \in \Lambda.
$$
Using the convolution theorem for the Fourier transform, the latter relation is equivalent to the property that the $n$-fold convolution
$$
p \coloneqq p_1 * \dots * p_n
$$
satisfies
$$
\ft p(\lambda) = \langle p, e_\lambda \rangle =0, \quad \lambda \in \Lambda.
$$
Notice that the support of $p$ is contained in
$$
[0,1] + [0,1] + \dots + [0,1] = [0,n],
$$
and moreover, $p \in L^2[0,n]$ by Young's convolution theorem. Since $D(\Lambda) \geq n$, it follows from Fourier uniqueness that $p=0$. Applying the Fourier transform shows that
$
\ft p = \prod_{j=1}^n \ft p_j
$
vanishes identically. Since $p_j$ has compact support, each factor $\ft p_j$ is analytic. Hence, one of the factors must vanish identically which means that there exists $k \in \{0, \dots, n-1 \}$ so that $p_k = f- \omega_k h =0$. This shows that $E(\Lambda)$ does $\Theta$-PR in $L^2[0,1]$.

\textbf{Necessity.}
    Let $\Lambda$ be a lattice of density $D(\Lambda) < n$, i.e., $\Lambda = \alpha \Z$ for some $\alpha > n$. We prove that $E(\Lambda)$ does not do $\Theta$-PR in $L^2[0,1]$. In order to prove this, we show that the assumptions of Corollary \ref{cor:PW_recurrence} are satisfied.
    
    To do so, we define $\Lambda_j = \alpha(n\Z+j)$ for $j=0,...,n-1$, which implies that $\{ \Lambda_j \}_{j=0}^{n-1}$ is a cover of $\Lambda$. Next, we construct the functions $x_j$. To do so, let
    $$
    \xi \coloneqq \frac{1}{2n\alpha}.
    $$
    Since $2\leq n<\alpha < \infty$, it holds that $0 < \xi < \frac{1}{2}$. Further, define functions $S_j$ via
    $$
    S_j(x) = \sin \left ( 2\pi \xi x - \frac{j\pi}{n} \right ), \quad j \in \{ 0, \dots, n-1 \}.
    $$
    Clearly, each $S_j$ vanishes on $\Lambda_j$. Using the trigonometric identity
    $$
    \sin(v-v') + \sin(v+v') = 2 \sin(v)\cos(v'), \quad v,v' \in \R,
    $$
    and substituting $v = 2\pi \xi x  - (j-1)\frac{\pi}{n}$ and $v' = \frac{\pi}{n}$, it follows that
    $$
    S_j = 2 \cos(\tfrac{\pi}{n}) S_{j-1} - S_{j-2}, \quad j \geq 3
    $$
    Next, choose $\Phi \in PW$ such that $\hat \Phi$ has support in $[-(\frac{1}{2}-\xi),\frac{1}{2}-\xi]$, and $\Phi(0) \neq 0$. To obtain such a function, we can take $\phi \in L^2(\R) \setminus \{0 \}$ satisfying
    $$
    \supp(\phi) \subseteq [-(\tfrac{1}{2}-\xi),\tfrac{1}{2}-\xi], \quad \int_\R \phi(x) \, dx \neq 0,
    $$
    and define $\Phi \coloneqq \ift \phi$.
    Since the sine-function $S_j$ can be written as a linear combination of the exponentials $e^{2\pi i \xi x}$ and $e^{-2\pi i \xi x}$,
    $$
    S_j(x) = C e^{2\pi i \xi x} + C'e^{-2\pi i \xi x}, \quad C=-\frac{i}{2}\, e^{ -\frac{i j \pi}{n}}, \quad C' = \frac{i}{2}\, e^{\frac{i j \pi}{n}},
    $$
    it follows that multiplying $\Phi$ with $S_j$ shifts the support of $\mathcal{F}(S_j \Phi)$ by $\pm \xi$. This and the choice of $\Phi$ implies that $S_j \Phi$ is an element of $PW$. It also holds that $S_j \Phi$ is not the zero function: if it would be zero, then the product of the two analytic functions $C e^{2\pi i \xi x} + C'e^{-2\pi i \xi x}$ and $\Phi$ would vanish identically. Since $\Phi \neq 0$, it follows that $C e^{2\pi i \xi x} + C'e^{-2\pi i \xi x}$ is the zero function. Using the linear independence of complex exponentials, we have $C=C'=0$, which gives a contradiction.
    
Now let $\omega \coloneqq e^{\frac{2\pi i }{n}}$, $\zeta \coloneqq e^{\frac{2\pi i }{2n}}$, and define $x_j \in PW$ via
    $$
    x_j \coloneqq \zeta^j S_j \Phi.
    $$
    Then $x_j \neq 0$ and $x_j$ vanishes on $\Lambda_j$.
    Using the recurrence relation for $S_j$ it follows from a direct calculation that
    \begin{equation}\label{eq:Tj_relation}
        x_j = 2 \cos(\tfrac{\pi}{n}) \zeta \cdot  x_{j-1} - \zeta^2 \cdot  x_{j-2}.
    \end{equation}
    By Euler's identity we have
    $$
    2 \cos(\tfrac{\pi}{n}) \zeta = 2 \frac{e^{i\frac{\pi}{n}} + e^{-i\frac{\pi}{n}}}{2} e^{\frac{2\pi i }{2n}} = 1+\omega.
    $$
    Combining this with $\zeta^2 = \omega$, shows that \eqref{eq:Tj_relation} can be re-written as
    $$
    x_j = (1+\omega) x_{j-1} - \omega x_{j-2}.
    $$
\end{proof}

Our final goal for this subsection is to prove Theorem \ref{thm:ER_disconnected}. To do so, we make use of the following Lemma.

\begin{lemma}\label{lma:real_line_onto_arc}
For $v_1 > v_2 > 0$ and $\beta \in \mathbb{R}$ define
$$
m(z) \coloneqq e^{i\beta} \frac{z + i v_1}{z - i v_1} \frac{z - i v_2}{z + i v_2}.
$$
Then $m$ maps the real line onto the arc $\left\{ e^{it}:  t\in[\beta-L,\beta+L] \right\}$, where $L>0$ is given by
$$
L = 4\arctan\sqrt{\frac{v_1}{v_2}}-\pi.
$$
\end{lemma}
\begin{proof}
    Fix $x \in \R$. As a preliminary step, it is easy to see that for every $v>0$ we have $\frac{x+iv}{x-iv}\in \T$. Writing $x + i v = r e^{i\theta}$ with $r = \sqrt{x^2 + v^2}$ and $\theta = \arctan \frac{v}{x},$ we get
$$
\frac{x + i v}{x - i v}
%= \frac{r e^{i\theta}}{r e^{-i\theta}}
%= e^{2 i \theta}
= e^{2 i \arctan(v/x)}.
$$
Consider the function $m(z)$ defined as in the statement of the lemma. Using the preliminary step for $v_1>v_2>0$ gives
$$
m(x)
= e^{i ( \beta + 2\arctan(v_1/x) - 2\arctan(v_2/x))}.
$$
Therefore
$$
\arg m(x)
= \beta
  + 2\arctan\!\frac{v_1}{x}
  - 2\arctan\!\frac{v_2}{x}.
$$
Now define
$$
\Phi(x) = 2\arctan\!\frac{v_1}{x}
  - 2\arctan\!\frac{v_2}{x}.
$$
We also put $\Phi(0):=0$. 
Then $\Phi$ is continuous on $\R$ with $\Phi(x) \to 0$ as $|x| \to \infty$. It suffices to show that $\max_{x\in \R}\Phi(x)=L$ and $\min_{x\in \R}\Phi(x)=-L$, from which we can deduce $\Phi(\R)=[-L,L]$, hence $\arg(m(x))=[\beta-L, \beta+L]$. An elementary calculation shows that
$$
\frac{d}{dx} \Phi(x) = \frac{2\,(v_1-v_2)\,(v_1 v_2 - x^2)}
{(x^2+v_1^2)(x^2+v_2^2)}.
$$
Since $v_1 > v_2 > 0$, it follows that $\Phi$ is strictly increasing on $(-\sqrt{v_1 v_2},\sqrt{v_1 v_2})$ and strictly decreasing on each of $(-\infty,-\sqrt{v_1 v_2})$ and $(\sqrt{v_1 v_2},\infty)$. Moreover, the only critical points of $\Phi$ are $\pm\sqrt{v_1 v_2}$. Since $\Phi$ is an odd function, we see that $\sqrt{v_1 v_2}$ yields the global maximum of $\Phi$ and $-\sqrt{v_1 v_2}$ the global minimum.

At $x=\sqrt{v_1 v_2}$, we have
$$
\Phi \left ( \sqrt{v_1 v_2} \right )
=2\left ( \arctan\sqrt{\frac{v_1}{v_2}}-\arctan\sqrt{\frac{v_2}{v_1}} \right ).
$$
For $\sigma>0$ the identity $\arctan \sigma+\arctan(1/\sigma)=\pi/2$ gives
$$
\arctan\sqrt{\frac{v_1}{v_2}}-\arctan\sqrt{\frac{v_2}{v_1}}
=2\arctan\sqrt{\frac{v_1}{v_2}}-\frac{\pi}{2},
$$
whence
$$
\Phi\bigl(\sqrt{v_1 v_2}\bigr)
=4\arctan\sqrt{\frac{v_1}{v_2}}-\pi
= L.
$$
By oddness, $\Phi(-\sqrt{v_1 v_2}) = -L$, which gives that $\Phi(\R)=[-L,L]$ as claimed.
\end{proof}

As a second ingredient for the proof of Theorem \ref{thm:ER_disconnected}, we require the classical Paley-Wiener theorem which says that a function $f$ belongs to the Paley-Wiener space $PW$ if and only if $f \in \lt$ and $f$ is the restriction to $\R$ of an entire function $F : \C \to \C$ of exponential type at most $\pi$ \cite{Young}. Recall that the latter means that for every $\varepsilon>0$ there exists $C_\varepsilon>0$ such that
\begin{equation}\label{eq:exp_type}
    |f(z)| \leq C_\varepsilon e^{(\pi+\varepsilon)|z|}, \quad z \in \C.
\end{equation}
We also observe that using equation \eqref{eq:reformulation} with $\Lambda = \R$, Theorem \ref{thm:ER_disconnected} admits the following reformulation, which will turn out more convenient to prove. Namely, $\Theta$ is totally disconnected if and only if for every $f,h \in PW$ the following implication holds:
\begin{equation}\label{eq:pw_line_impl}
    \forall t \in \R \, \exists \theta_t \in \Theta \ \text{such that} \ f(t) = \theta_t h(t) \implies f=\theta h \text{ for some } \theta\in\Theta.
\end{equation}

\begin{proof}[Proof of Theorem \ref{thm:ER_disconnected}]
\textbf{Necessity.} Suppose that $\Theta$ is not totally disconnected. We construct $f,h \in PW$ so that the implication \eqref{eq:pw_line_impl} fails.

To do so, observe that if $\Theta$ is not totally disconnected then there exist $\tilde \varepsilon>0$ and $\tilde \beta \in \R$ such that
\begin{equation}\label{eq:arc_inclusion}
    A \coloneqq \{ e^{i(\tilde\beta+t)} : t \in [-\tilde\varepsilon,\tilde\varepsilon] \} \subseteq \Theta.
\end{equation}
We choose $v_1>v_2>0$ such that
$$
0 < 4 \arctan \frac{v_1}{v_2} - \pi \leq \tilde\varepsilon.
$$
Let $s \in C_c^\infty(\R)$ be a smooth function with support in $[-\frac{1}{2},\frac{1}{2}]$ and put $S := \ift s$. Then $S$ is a Schwartz function that belongs to $PW$. In particular $S$ has an analytic extension to the complex plane. Now define
$$
f: \C \to \C, \quad f(z) = (z+iv_1)(z-iv_2) S(z).
$$
Clearly, $f$ is again a Schwartz function when considered as a function on the real line. Moreover, the prefactor $P(z) = (z+iv_1)(z-iv_2)$ does not affect the exponential type of $PS$ in the sense that both $PS$ and $S$ satisfy the groth estimte \eqref{eq:exp_type}. By the Paley-Wiener theorem, we have $f \in PW$. Let $m(z)$ be the function from Lemma \ref{lma:real_line_onto_arc} with $\varepsilon \coloneqq \tilde \varepsilon$ and $\beta \coloneqq \tilde \beta$. Observe that $f$ has the property that the set of poles of $m$ is a subset of the set of zeros of $f$. This allows us to define a second Schwartz function $h$ via
$$
h(z) \coloneqq \frac{1}{m(z)} f(z) = e^{-i \beta} (z-i v_1) (z+iv_2) S(z).
$$
By an analogous argument as for the function $f$, it follows that $h \in PW$. Thus, $f$ and $h$ are related via
$$
f=mh
$$
By Lemma \ref{lma:real_line_onto_arc} and the inclusion \eqref{eq:arc_inclusion}, it follows that $m(t) \in \Theta$ for every $t \in \R$. Since the quotient $f/h$ is not a constant function, it follows that $f \neq \theta h$ for every $\theta \in \Theta$. This proves the necessity part of the statement.

\textbf{Sufficiency.} It remains to show that if $\Theta$ is totally disconnected, then \eqref{eq:pw_line_impl} holds true. Therefore, let $f,h \in PW$ such that for every $t\in \R$ there exists $\theta_t\in \Theta$ with
    \begin{equation}\label{eq:assumption_lt_mathcalO}
        f(t) = \theta_t h(t).
    \end{equation}
    We have to show that there exists $\theta \in \Theta$ such that $f=\theta h$.

    If $f=0$ then also $h=0$ and we are done. If $f$ does not vanish identically, then there exists $t \in \R$ such that $f(t) \neq 0$. Since $f$ extends to an entire function, there exists an open disc $D \subseteq \C$ around $t$ such that $f$ does not vanish on $D$.

    As $f(t) \neq 0$, it follows from \eqref{eq:assumption_lt_mathcalO} that $h(t) \neq 0$. Hence, there exists a second disc $D'\subseteq \C$ around $t$ such that $h$ does not vanish on $D'$ (the disc $D'$ can be possibly smaller that $D$, since $h$ can have a zero in $D \setminus \R$).
    In particular, both $f,h$ do not vanish on the intersection
    $$
    D'' \coloneqq D \cap D',
    $$
    and the quotient $f/h$ is holomorphic in $D''$. 
    
    We will now show that $f/h$ must be constant on $D''\cap \R$. Suppose not. The continuity of $f/h$ on $D'' \cap \R$ (which is a proper interval, hence a connected set) implies that the image of $D'' \cap \R$ under $f/h$ must be connected and contain a proper arc. Since
    $$
    \left \{ \frac{f(t)}{h(t)} : t \in D'' \cap \R \right \} = \{ \theta_t : t \in D'' \cap \R \} \subseteq \Theta,
    $$
    it follows that $\Theta$ contains a connected set, a contradiction to the assumption on $\Theta$. 

    Thus $f/h$ is constant on $D''\cap \R$. It follows from \eqref{eq:assumption_lt_mathcalO}, that there exists $\theta \in \Theta$ such that
    $$
    \frac{f(t)}{h(t)} = \theta, \quad t \in D'' \cap \R.
    $$
    Hence, $f(t) = \theta h(t)$ for all $t \in D'' \cap \R$. Since $D'' \cap \R$ is a proper interval and $f,h$ are holomorphic on the entire complex plane, it follows from the uniqueness theorem of holomorphic functions that $f=\theta h$.
\end{proof}

\section{Möbius invariance and cross ratios}\label{sec:moebius}

The purpose of this section is to develop a systematic connection between classical topics in projective geometry - most notably Möbius transforms and cross ratios - and the theory of $\Theta$-PR. In particular, we show that Möbius transforms play a central role in analyzing $\Theta$-PR. Both Möbius transforms and cross ratios serve as essential tools in the proofs of Corollaries \ref{cor:4_pr_and_cr} and \ref{cor:pr_C2_characterization}, as well as Proposition \ref{prop:arc}.

\subsection{Möbius invariance}

Recall that for an invertible matrix $A = \begin{pmatrix}
    a & b \\ c & d
\end{pmatrix}$ with complex entries $a,b,c,d \in \C$, the Möbius transform corresponding to $A$ is defined by
$$
M_A(z) = \frac{az+b}{cz+d}.
$$
Let $\U(1,1)$ denote the group of all matrices $A \in \mathrm{GL}(2,\C)$ such that
$$
A^*JA = J, \quad J = \begin{pmatrix}
    1 & 0 \\ 0 & -1
\end{pmatrix}.
$$
If $A \in \U(1,1)$, then $M_A\in\mathrm{Aut}(\T)$, where $\mathrm{Aut}(\T)$ denotes the group of bijections on $\T$. Conversely, every Möbius transform $M\in \mathrm{Aut}(\T)$ comes from $A \in \U(1,1)$ up to multiplication of $A$ by a unimodular scalar \cite[Chapter 10.4]{simon2005orthogonal}.
For a given set $\Theta \subseteq \T$ and a Möbius transform $M\in \mathrm{Aut}(\T)$, we denote by $M(\Theta)$ the image of $\Theta$ under $M$, i.e.,
$
M(\Theta) = \{ M(\theta) : \theta \in \Theta \}.
$
We are ready to prove Theorem \ref{thm:mbs}.

\begin{proof}[Proof of Theorem \ref{thm:mbs}]
    The implications $(1) \implies (2)$ and $(3) \implies (1)$ follow from the fact that the identity map is a Möbius transform. Therefore, the three statements become equivalent once we have shown that $(2) \implies (3)$.
    
    We show that $(2) \implies (3)$ via contraposition: assume that there exists a Möbius transform $M\in \text{Aut}(\T)$ such that $\mathbf{G}$ fails $M(\Theta)$-PR. For a fixed $M'\in \text{Aut}(\T)$ we will show that $\mathbf{G}$ also fails $M'(\Theta)$-PR.
    
    From failure of $M(\Theta)$-PR and Lemma \ref{lma:failure_witnessing}, we can find linearly independent $f,h \in H$ such that for every $g\in \mathbf{G}$ there exists $\theta_g \in M(\Theta)$ with 
    $$
    \langle f,g\rangle = \theta_g \langle h,g\rangle.
    $$
    Let $M_0\in \text{Aut}(\T)$ be a Möbius transform of the form
    $$
    M_0(z) = \frac{az+b}{cz+d},
    $$
    with $ad-bc\neq 0$ ($M_0$ will be specified later). Define $x,y \in H$ via
    $$
    x = af+bh, \quad y = cf+dh.
    $$
    Since $f$ and $h$ are linearly independent and since $ad-bc\neq 0$ it follows that $x$ and $y$ are linearly independent. We will show that for every $g\in \mathbf{G}$ there exists $\theta_g\in (M_0 \circ M)(\Theta)$ such that 
    $$\langle x,g\rangle = \theta_g\langle y,g\rangle.$$
    Fix $g\in \mathbf{G}$. If $\langle h,g\rangle=0$, then also $\langle f,g\rangle=0$, hence $\langle x,g\rangle = 0=\theta_g\langle y,g\rangle$ for any $\theta_g \in (M_0 \circ M)(\Theta)$.
    Suppose $\langle f,g\rangle,\langle h,g\rangle\neq 0$. Then the quotient
    $$
    \frac{\langle f,g\rangle}{\langle h,g\rangle}
    $$
    is a well-defined complex number on the unit circle which satisfies 
    $$M_0\Big(\frac{\langle f,g\rangle}{\langle h,g\rangle}\Big) = \frac{a\frac{\langle f,g\rangle}{\langle h,g\rangle}+b}{c\frac{\langle f,g\rangle}{\langle h,g\rangle}+d}=\frac{a\langle f,g\rangle + b\langle h,g\rangle}{c\langle f,g\rangle + d\langle h,g\rangle} = \frac{\langle x,g\rangle}{\langle y,g\rangle}.$$
    Thus, we obtain
    $$\frac{\langle x,g\rangle}{\langle y,g\rangle} = M_0(\theta_g) \in (M_0 \circ M)(\Theta).$$ 
    Since $x,y$ are linearly independent, it follows that $\mathbf{G}$ fails $(M_0 \circ M)(\Theta)$-PR. Since $M_0$ was arbitrary, picking $M_0:=M'\circ M^{-1}$ implies that $\mathbf{G}$ fails $M'(\Theta)$-PR.
\end{proof}

\subsection{Relation to cross ratios}

Our next goal is to establish a relation between $\Theta$-PR and projective invariants. To do so, we first recall some basic facts about cross ratios. For pairwise distinct $z_1,z_2,z_3,z_4 \in \C$, the cross ratio is defined by
$$
\cratio(z_1,z_2;z_3,z_4) = \frac{(z_1-z_3)(z_2-z_4)}{(z_1-z_4)(z_2-z_3)}.
$$
It is known that the cross ratio is invariant under Möbius transforms, in the sense that for every Möbius transform $M$ and every pairwise distinct $z_1,z_2,z_3,z_4 \in \C$ we have
$$
\cratio(M(z_1),M(z_2);M(z_3),M(z_4)) = \cratio(z_1,z_2;z_3,z_4).
$$
Finally, the cross ratio $\cratio(z_1,z_2;z_3,z_4)$ is real if and only if $z_1,z_2,z_3,z_4 \in \C$ are four distinct points on the same circle or line. The cross ratio can be used to give a criterion under which two quadruples are related by a Möbius transform: if $\{ \theta_1,\theta_2,\theta_3,\theta_4\} \subseteq \T$ and $\{ \theta_1',\theta_2',\theta_3',\theta_4'\} \subseteq \T$ are two quadruples then $$\cratio(\theta_1,\theta_2;\theta_3,\theta_4) = \cratio(\theta_1',\theta_2';\theta_3',\theta_4')$$ if and only if there exists a Möbius transform $M\in \text{Aut}(\T)$ such that $M(\theta_j)=\theta_j'$ for every $j \in \{ 1,2,3,4 \}$ (see for example \cite{hohl2025laplace} for an explicit proof). We are ready to prove the following statement.

\begin{corollary}\label{cor:4_pr_and_cr}
    Let $H$ be a Hilbert space, let $\mathbf{G}\subseteq H$ and let $\Theta,\Theta'\subseteq \T$ satisfy $|\Theta|=|\Theta'|=4$. If $\cratio(\Theta)=\cratio(\Theta')$, then $\mathbf{G}$ does $\Theta$-PR if and only if $\mathbf{G}$ does $\Theta'$-PR. 
\end{corollary}
\begin{proof}
    Suppose $\Theta,\Theta'\subseteq \T$ satisfy $\cratio(\Theta)=\cratio(\Theta')$. This implies that there exists a Möbius transform $M\in\text{Aut}(\T)$ such that $M(\Theta)=\Theta'$. From Theorem \ref{thm:mbs} we get that $\mathbf{G}$ does $\Theta$-PR if and only if $\mathbf{G}$ does $\Theta'$-PR.
\end{proof}

\begin{example}
    The reverse implication in the latter corollary is in general not true. Indeed, take any system $\mathbf{G}$ doing PR. From Lemma \ref{lemma:theta_pr_implies_theta'_pr}, such $\mathbf{G}$ does $\Theta$-PR for any $\Theta\subseteq \T$. In particular, we can pick any $\Theta,\Theta'\subseteq \T$ satisfying $|\Theta|=|\Theta'|=4$ and $\cratio(\Theta)\neq\cratio(\Theta')$.
\end{example}

\section{$\Theta$-PR in $\Cd$}

In this section we study $\Theta$-PR in the finite dimensional Hilbert space $H=\C^d$. Note that if $d=1$ then for any $\Theta\subseteq \T$, $\mathbf{G}$ does $\Theta$-PR if and only if $\mathbf{G}\neq \{0\}$. We therefore restrict our attention to the case $d\geq 2$.

\subsection{$\Theta$-PR in $\C^2$}\label{sec:c2_characterization}

The first nontrivial setting in which $\Theta$-PR can be analyzed is the space $\C^2$. In this case, the so-called $4d - 4$ conjecture is known to hold: there exist systems of four vectors that do phase retrieval in $\C^2$, while no system with fewer than four vectors does \cite[Theorem 10]{bandeira2014saving}. Moreover, the collection of complete systems consisting of four-element doing phase retrieval in $\C^2$ has been explicitly characterized in \cite{grohs2025multi}. It is shown there that, up to an appropriate linear change of coordinates, the problem reduces to studying systems of the form
\begin{equation}
        \mathbf{G}(a,b,c) =\left \{  
             \begin{pmatrix}
               1 \\
               0 
             \end{pmatrix},    
             \begin{pmatrix}
               a \\
               1 
               \end{pmatrix},
               \begin{pmatrix}
               b \\
               1 
               \end{pmatrix},\begin{pmatrix}
               c \\
               1 
               \end{pmatrix}
               \right \},
\end{equation}
where $a,b,c \in \C$. The failure of phase retrieval for such a system admits a geometric characterization \cite[Theorem 1.4]{grohs2025multi}: $\mathbf{G}(a,b,c)$ fails phase retrieval in $\C^2$ if and only if $a,b,c$ are collinear. To the best of our knowledge, $\C^2$ remains the only Hilbert space for which such a concrete and complete characterization of phase retrieval is available. In the present subsection, we study the systems $\mathbf{G}(a,b,c)$ in the context of $\Theta$-PR; we note that the same technique used to reduce the study of PR of $\mathbf{G}\subseteq \C^2$ to systems of the form $\mathbf{G}(a,b,c)$ also works for $\Theta$-PR due to Lemma \ref{lemma:invariantTransform}. We also point out that it is not needed to characterize $\Theta$-PR for every $\Theta\subseteq\T$. Indeed, straight from the definition of PR and $\Theta$-PR we obtain the following.

\begin{lemma}\label{lma:m_PR_is_PR}
    Let $\mathbf{G}\subseteq \C^d$ satisfy $|\mathbf{G}|=m$. Then $\mathbf{G}$ does PR if and only if for every $\Theta\subseteq \T$, with $|\Theta|\leq m$, $\mathbf{G}$ does $\Theta$-PR.
\end{lemma}

The latter lemma essentially says that in order to understand PR for a system $\mathbf{G}$ it suffices to study $\Theta$-PR for $|\Theta|\leq |\mathbf{G}|$. Since in this section we restrict our attention to $\C^2$ and systems $\mathbf{G}(a,b,c)$, it suffices to analyze whenever $\mathbf{G}(a,b,c)$ does $\Theta$-PR for $|\Theta|\leq 4$.

\begin{proposition}\label{prop:C2framesCharacterization}
    Let $a,b,c \in \C$. Then the following holds:
    \begin{enumerate}
        \item $\mathbf{G}(a,b,c)$ fails $2$-PR if and only if $a=b=c$.
        \item $\mathbf{G}(a,b,c)$ fails $3$-PR if and only if $a=b$ or $b=c$, or $a=c$.
        \item $\mathbf{G}(a,b,c)$ fails $\Theta$-PR with $|\Theta|=4$ if and only if $a=b$ or $b=c$, or $a=c$, or there exists an order $(\theta_1, \theta_2,\theta_3, \theta_4)$ of $\Theta$ such that
        $$\frac{c-a}{b-a} = \overline{\mathrm{CR}(\theta_1,\theta_2;\theta_3,\theta_4)}.
        $$
    \end{enumerate}
\end{proposition}
\begin{proof}
    (1). According to Proposition \ref{prop:cp} it suffices to show that $\mathbf{G}$ fails the complement property if and only if $a=b=c$.

    Suppose $\mathbf{G}$ fails the complement property. We can find $\mathbf{S}\subseteq \mathbf{G}$ such that $\mathbf{S}$ nor $\mathbf{G\setminus S}$ are complete. A quick elimination of potential cases shows that the only option when this is possible is when $\mathbf{S} = \{(1,0)\}$ and $\mathbf{G\setminus S} = \{(a,1),(b,1),(c,1)\}$. As $\mathbf{G\setminus S}$ is not complete, this forces $a=b=c$.

    If $a=b=c$ we get failure of the complement property for $\mathbf{S} = \{(1,0)\}$.
    
    (2). We will use the characterization from Corollary \ref{cor:3_pr_characterization}. Suppose $a=b$. Consider $\mathbf{G_1} = \{(1,0)\}$, $\mathbf{G_2} = \{(a,1),(b,1)\}=\{(a,1)\}$ and $\mathbf{G_3} = \{(c,1)\}$. Then by taking $x_1 = (0,1)$, $x_2 = (1, -\overline{a})$ and $x_3 = (1, -\overline{c})$ we get that $x_j\in \mathbf{G}_j^\perp$ for $j=1,2,3$. Moreover, since $x_j \in \C^2$, the system $x_3 = Ax_1+Bx_2$ has a non-zero solution $(A,B)$. The cases when $a=c$ or $b=c$ are proved similarly.
    
    Now, suppose there exists a cover $\mathbf{G}_1,\mathbf{G}_2,\mathbf{G}_3$ or $\mathbf{G}$ and non-zero $x_j\in\mathbf{G}_j^\perp$ with $x_1,x_2$ linearly independent and $x_3\in \lspan\{x_1,x_2\}$. In particular this means that $\mathbf{G}_j$ for $j=1,2,3$ is not complete. Since the sets $\{(1,0),(a,1)\}$, $\{(1,0),(b,1)\}$ and $\{(1,0),(c,1)\}$ are always complete, the only option for the case where all three subsets are incomplete forces $\mathbf{G}_1 = \{(1,0)\}$. If say $\mathbf{G}_2 = \{(a,1),(b,1)\}$, then incompleteness of $\mathbf{G}_2$ gives $a=b$. The other cases give $a=b$ or $a=c$.

    (3). Fix $\Theta$ with $|\Theta|=4$. We first reduce this problem further. Consider the invertible matrix $T \in \C^{2 \times 2}$ and constants $p,q \in \C$ defined by
    $$
    T = \begin{pmatrix}
        1 & -a \\
        0 & 1
    \end{pmatrix}, \quad p = b-a, \quad q = c-a.
    $$
    The image of $\mathbf{G}(a,b,c)$ under $T$ satisfies
    $$
    T(\mathbf{G}(a,b,c)) = \tilde{\mathbf{G}}(p,q) \coloneqq \left \{  
             \begin{pmatrix}
               1 \\
               0 
             \end{pmatrix},    
             \begin{pmatrix}
               0 \\
               1 
               \end{pmatrix},
               \begin{pmatrix}
               p \\
               1 
               \end{pmatrix},\begin{pmatrix}
               q \\
               1 
               \end{pmatrix}
               \right \}=\{g_1,g_2,g_3,g_4\}.
    $$
    According to Lemma \ref{lemma:invariantTransform}, it's enough to show that $\tilde{\mathbf{G}}(p,q)$ fails $\Theta$-PR if and only if $p=0$ or $q=0$ or $p=q$ or there exists an order $(\theta_1,\theta_2,\theta_3,\theta_4)$ of $\Theta$ such that  $\overline{\mathrm{CR}(\theta_1,\theta_2;\theta_3,\theta_4)}=\frac{q}{p}$.
    
    For the forward direction, suppose $\mathbf{G}$ fails $\Theta$-PR. If additionally $\tilde{\mathbf{G}}(p,q)$ fails $3$-PR, then $p=0$ or $q=0$, or $p=q$ from part (b). Suppose $\tilde{\mathbf{G}}(p,q)$ does $3$-PR. This means we find $f,h\in H$ and an order $(\theta_1,\theta_2,\theta_3,\theta_4)$ of $\Theta$ such that $f\neq \theta h$ for $\theta \in \Theta$ and yet $\langle f,g_j\rangle = \theta_j\langle h,g_j\rangle$ for $j=1,2,3,4$. Using the definition of $\tilde{\mathbf{G}}(p,q)$ we get
    \begin{align}
        f_1 &= \theta_1 h_1\\
        f_2 &=\theta_2 h_2\\
        \overline{p}f_1 + f_2 &= \theta_3(\overline{p}h_1+h_2)\\
        \overline{q}f_1 + f_2 &= \theta_4(\overline{q}h_1+h_2).\\
    \end{align} 
    Reorganizing it gives 
        \begin{align}
        \overline{p}h_1(\theta_1-\theta_3) + h_2(\theta_2-
        \theta_3)&= 0\\
        \overline{q}h_1(\theta_1-\theta_4) + h_2(\theta_2-
        \theta_4)&= 0.\\
    \end{align}
    Note, if $h_1=0$ and $h_2\neq 0$ we get $f_1=0$ and since $f_2=\theta_2h_2$ this gives $f=\theta_2h$ contradicting failure of $\Theta$-PR. Hence $h_1\neq 0$. Similar argument shows that $h_2\neq 0$. Solving the system of equations gives 
            $$\frac{q}{p} = \overline{\frac{(\theta_1-\theta_3)(\theta_2-\theta_4)}{(\theta_2-\theta_3)(\theta_1-\theta_4)}} = \overline{\cratio(\theta_1,\theta_2;\theta_3,\theta_4)}.$$ 

    For the reverse implication, if $p=0$ or $q=0$ or $p=q$ then $\tilde{\mathbf{G}}(p,q)$ fails $3$-PR from part (b), hence $\Theta$-PR from Lemma \ref{lemma:theta_pr_implies_theta'_pr}. Assume that $p,q\neq 0$. Suppose there exists an order $(\theta_1,\theta_2,\theta_3,\theta_4)$ of $\Theta$ such that $\frac{q}{p}=\overline{\cratio(\theta_1,\theta_2;\theta_3,\theta_4)}$. Take
          \begin{align*}
            f=\begin{pmatrix}
            \frac{\theta_1}{\overline{p}}\\
            -\theta_2\frac{\theta_1-\theta_3}{\theta_2-\theta_3}
            \end{pmatrix}, \quad
            h=\begin{pmatrix}
            \frac{1}{\overline{p}}\\
            -\frac{\theta_1-\theta_3}{\theta_2-\theta_3}
            \end{pmatrix}.
            \end{align*}
    Since elements of $\Theta$ are pairwise distinct, we have $f\neq \theta h$ for any $\theta \in \Theta$. Next, we immediately get that $\langle f,g_j\rangle = \theta_j\langle h,g_j\rangle$ for $j=1,2$. To show that $\langle f,g_3\rangle=\theta_3\langle h,g_3\rangle$ we first note that a direct calculation shows 
    $$\theta_1-\theta_2\frac{\theta_1-\theta_3}{\theta_2-\theta_3} = \theta_3\Big(1-\frac{\theta_1-\theta_3}{\theta_2-\theta_3}\Big).$$
    This gives
    \begin{align*}
        \langle f,g_3\rangle = \theta_1-\theta_2\frac{\theta_1-\theta_3}{\theta_2-\theta_3} = \theta_3\Big(1-\frac{\theta_1-\theta_3}{\theta_2-\theta_3}\Big) = \theta_3\langle h,g_3\rangle. 
    \end{align*}
    Finally, to show that $\langle f,g_4\rangle = \theta_4\langle h,g_4\rangle$ we observe that 
    $$\theta_2 - \theta_1\frac{\theta_2-\theta_4}{\theta_1-\theta_4} = \theta_4\Big(1-\frac{\theta_2-\theta_4}{\theta_1-\theta_4}\Big).$$
    Using the assumption that $\frac{\overline{q}}{\overline{p}}=\cratio(\Theta)$ we get
    \begin{align*}
        \langle f,g_4\rangle &= \theta_1 \frac{\overline{q}}{\overline{p}} - \theta_2 \frac{\theta_1-\theta_3}{\theta_2-\theta_3} = \theta_1 \frac{(\theta_1-\theta_3)(\theta_2-\theta_4)}{(\theta_2-\theta_3)(\theta_1-\theta_4)} - \theta_2\frac{\theta_1-\theta_3}{\theta_2-\theta_3}\\
        &=-\frac{\theta_1-\theta_3}{\theta_2-\theta_3}\Big(\theta_2 - \theta_1\frac{\theta_2-\theta_4}{\theta_1-\theta_4}\Big) = -\frac{\theta_1-\theta_3}{\theta_2-\theta_3}\theta_4\Big(1-\frac{\theta_2-\theta_4}{\theta_1-\theta_4}\Big)\\
        &=\theta_4 \Big(\frac{\overline{q}}{\overline{p}} - \frac{\theta_1-\theta_3}{\theta_2-\theta_3}\Big) = \theta_4\langle h,g_4\rangle.
    \end{align*}
\end{proof}

Using characterization from Proposition \ref{prop:C2framesCharacterization} and the observation that cross ratio of elements from the unit circle is a real number, we obtain a different method of characterizing PR in $\C^2$, equivalent to the one obtained in \cite[Theorem 1.4]{grohs2025multi}.

\begin{corollary}\label{cor:pr_C2_characterization}
    Let $\mathbf{G}=\{(1,0), (a,1), (b,1), (c,1)\}\subseteq \C^2$. Then $\mathbf{G}$ fails PR if and only if $b=a$ or $\frac{c-a}{b-a}\in \R$.
\end{corollary}
\begin{proof}
    \textbf{Necessity}. Suppose $\mathbf{G}$ fails PR. From the definition of PR we can find $f,h\in \C^2$ and $\Theta=\{\theta_1,\theta_2,\theta_3,\theta_4\}$ witnessing this failure. From Proposition \ref{prop:C2framesCharacterization} part (c) this forces that $a=b$ or $a=c$, or $b=c$ (the latter two giving $\frac{c-a}{b-a}\in \R$), or         
    $\frac{c-a}{b-a} = \overline{\cratio(\Theta)}$. Since $\cratio(\Theta)$ is always a real number, we get $\frac{c-a}{b-a}\in \R$.

    \textbf{Sufficiency}. If $a=b$ $\mathbf{G}$ fails $3$-PR from Proposition \ref{prop:C2framesCharacterization} part (b), thus PR from Lemma \ref{lemma:theta_pr_implies_theta'_pr}. Suppose $a\neq b$ and $\frac{c-a}{b-a}\in \R$. We will show that one can find $\Theta=\{\theta_1,\theta_2,\theta_3,\theta_4\}\subseteq \T$ such that $\frac{c-a}{b-a} = \overline{\cratio(\Theta)}$
    which will show failure of $\Theta$-PR from Proposition \ref{prop:C2framesCharacterization} part (c), hence failure of PR from Lemma \ref{lemma:theta_pr_implies_theta'_pr}. Denote $r=\frac{c-a}{b-a}\in \R$. Let $\alpha \in [0,2\pi]$ and choose $\theta_1=1, \theta_2=e^{i\alpha}, \theta_3 = e^{i(\alpha + \pi)}$ and $\theta_4 = e^{i(\alpha+\pi/2)}$. We will show that one can find $\alpha$ making the cross ratio equal $r$. A direct calculations gives
        \begin{align*}
        \frac{(\theta_2-\theta_4)(\theta_1 - \theta_3)}{(\theta_1-\theta_4)(\theta_2-\theta_3)}&=\frac{(1-e^{i(\alpha + \pi/2)})(e^{i\alpha} - e^{i(\alpha+\pi)})}{(1-e^{i(\alpha+\pi)})(e^{i\alpha} - e^{i(\alpha+\pi/2)})} = \frac{1-ie^{i\alpha}}{1+e^{i\alpha}}\frac{2e^{i\alpha}}{e^{i\alpha}-ie^{i\alpha}}\\
        &= \frac{1-ie^{i\alpha}}{1+e^{i\alpha}}(1+i) = \frac{1-ie^{i\alpha}+i+e^{i\alpha}}{1+e^{i\alpha}}= 1+i\frac{1-e^{i\alpha}}{1+e^{i\alpha}}\\ 
        &= 1+i\Big(1-\frac{2e^{i\alpha}}{1+e^{i\alpha}}\Big)= 1+\frac{\sin\alpha}{1+\cos\alpha} 
    \end{align*}
    As the function $1+\frac{\sin\alpha}{1+\cos\alpha}$ is onto $\R$ we get that there exists $\alpha$ such that $\cratio(\Theta)=r$, finishing the proof.
\end{proof}

Let us now take a closer look at the characterization obtained in Proposition \ref{prop:C2framesCharacterization}. When we regard the parameters $(a,b,c)$ defining the system $\mathbf{G}(a,b,c)$ as a point in $\C^3$, the failure of $\Theta$-PR can be described by explicit algebraic conditions on $(a,b,c)$. For example, $\mathbf{G}(a,b,c)$ fails $2$-PR if and only if $(a,b,c)$ lies in the zero set of the polynomial
$$
P : \C^3 \to \C, \quad P(x,y,z)=(x-y)^2+(y-z)^2+(x-z)^2.
$$
Similarly, $3$-PR can be fully characterized via the zero set of 
$$
Q : \C^3 \to \C, \quad Q(x,y,z) = (x-y)(y-z)(x-z).
$$
For the four-point case of $\Theta=\{\theta_1,\theta_2,\theta_3,\theta_4\}$, define
$$
R_\Theta : \C^3 \to \C, \quad R_\Theta(x,y,z) = (z-x) - \overline{\mathrm{CR}(\theta_1,\theta_2;\theta_3,\theta_4)} \cdot (y-x).
$$
Proposition \ref{prop:C2framesCharacterization} shows that $\mathbf{G}(a,b,c)$ fails $\Theta$-PR if and only if $(a,b,c)$ belongs to the zero set of $Q$ or $R_\Theta$. 
Since the zero set of a nontrivial polynomial is a proper algebraic subset of the ambient space, its complement is Zariski-open and therefore dense (and of full Lebesgue measure) in the Euclidean topology. This shows that a "generic" choice of $(a,b,c)$ implies that $\mathbf{G}(a,b,c)$ does $\Theta$-PR. If $\Theta$ consists of exactly two elements, then as pointed out in the introduction, every full-spark system does $\Theta$-PR and full-spark systems form again a non-empty Zariski-open set. In the next section, we make these observation precise and show that, in every dimension $d$, the property of doing $\Theta$-PR is generically satisfied.

\subsection{Failure of $\Theta$-PR in $\C^d$ for $d\geq 2$}\label{sec:failure}

To formulate a precise statement about generic $\Theta$-PR, we begin by identifying a condition under which $\Theta$-PR fails. This is the content of Theorem \ref{thm:failure}, which proof follows next.

\begin{proof}[Proof of Theorem \ref{thm:failure}]
    First, if $m\leq d-1$ then $\mathbf{G}$ is not complete. Thus, Lemma \ref{lma:completeness} implies that $\mathbf{G}$ fails $\Theta$-PR. Therefore we assume $m \geq d$.
    
    (1). If $d\leq m\leq 2d-2,$ we set $\mathbf{S} = \{g_1,...,g_{d-1}\}$. Since $\mathbf{S}$ and $\mathbf{G} \setminus \mathbf{S}$ cannot be complete in $\C^d$, this shows that $\mathbf{G}$ does not have the complement property. Hence from Proposition \ref{prop:cp} $\mathbf{G}$ fails $2$-PR, and thus by Lemma \ref{lemma:theta_pr_implies_theta'_pr} it also fails $\Theta$-PR for any $\Theta \subseteq \T$ with $|\Theta|\geq 2$.

    (2). Assume $m=2d-1$ and consider the cover $\mathbf{G}_1 = \{g_1,...,g_{d-1}\}$, $\mathbf{G}_2=\{g_d,...,g_{2d-2}\}$ and $\mathbf{G}_3 = \{g_{2d-1}\}$. Since for every $j \in \{ 1,2,3 \}$ it holds that $\mathbf{G}_j$ is not complete, there exists $x_1,x_2,x_3\in \C^d$ such that $x_j\in \mathbf{G}_j^\perp$.
    
    If for every such choice of $x_1,x_2$ it holds that $x_1$ and $x_2$ are linearly dependent, then $\lspan(\mathbf{G}_1) = \lspan(\mathbf{G}_2)$ and both systems $\mathbf{S} = \mathbf{G}_1 \cup \mathbf{G}_2$ and $\mathbf{G} \setminus \mathbf{S} = \mathbf{G}_3$ are incomplete. In particular, $\mathbf{G}$ does not have the complement property and therefore fails $2$-PR. By Lemma \ref{lemma:theta_pr_implies_theta'_pr}, this implies that $\mathbf{G}$ fails $\Theta$-PR for every $\Theta$ with $|\Theta| \geq 3$.

    It remains to consider the case when $x_1$ and $x_2$ can be chosen to be linearly independent. In this case, we show that the assumptions of Corollary \ref{cor:3_pr_characterization} are satisfied. To do so, it suffices to prove that $x_3\in \lspan\{x_1,x_2\}$, or equivalently, the existence of a solution to the equation
    $$a\langle x,g_{2d-1}\rangle + b\langle y,g_{2d-1}\rangle=0.$$
    If $\langle x,g_{2d-1}\rangle =0$, we can take $a=1$ and $b=0$. If $\langle x,g_{2d-1}\rangle \neq 0$ then the choices $b=1$ and $a=-\frac{\langle y,g_{2d-1}\rangle}{\langle x,g_{2d-1}\rangle}$ will do the job. From Corollary \ref{cor:3_pr_characterization} it follows that $\mathbf{G}$ fails $3$-PR, and thus by Lemma \ref{lemma:theta_pr_implies_theta'_pr} it also fails $\Theta$-PR for any $\Theta\subseteq \T$ with $|\Theta|\geq 3$.
    \end{proof}

Given $\Theta\subseteq \T$, we next investigate what one can expect if $\mathbf{G}\subseteq \C^d$ satisfies $|\mathbf{G}|\geq 2d$. Obviously, not every such $\mathbf{G}$ does $\Theta$-PR and it is natural to ask how likely it is for $\mathbf{G}$ with $|\mathbf{G}|\geq 2d$ to do $\Theta$-PR. We answer this question using the framework of Zariski topology. To do so, we start with several preliminary definitions.

\subsection{Generic $\Theta$-PR}\label{sec:zariski_topology}
Consider the space $\C^N$ with coordinates $x = (x_1,\dots,x_N)$. A polynomial in the variables $x_1,\dots,x_N$ is a finite linear combination of monomials $x^\alpha \coloneqq x_1^{\alpha_1} \cdots x_N^{\alpha_N}$ where $\alpha = (\alpha_1,\dots,\alpha_N) \in \N_{0}^N$. Let $\mathcal{P}$ denote the collection of all polynomials in $x_1,\dots,x_N$. Given $P \subseteq \mathcal{P}$, let
$$
V(P) = \{ x \in \C^N : p(x) = 0 \ \text{for every} \ p \in P \}
$$
denote the common zeros of the elements in $P$. A subset $C \subseteq \C^N$ is said to be Zariski-closed in $\C^N$ if there exists $P \subseteq \mathcal{P}$ such that $C=V(P)$. A set $O \subseteq \C^N$ is called Zariski-open if it is the complement of a Zariski-closed set.

According to Theorem \ref{thm:failure}, a system of $m$ elements does $\Theta$-PR in $\Cd$ for any $|\Theta| \geq 3$, only if $m \geq 2d$. The purpose of this section is to show that when $\Theta$ is finite, this necessary condition is in fact generically sufficient: among all systems of $m \geq 2d$ elements, those that do $\Theta$-PR contain a Zariski-open set. 

In order to prove the latter result, we require the following combinatorial lemma.

\begin{lemma}\label{lma:graph_theoretical}
    Let $\Theta\subseteq \T$ and $d\geq 2$. Further, let $\theta = (\theta_1, \dots, \theta_{2d}) \in \Theta^{2d}$ and let $s_1 \geq s_2 \geq \cdots \geq s_t \geq 1$ be the multiplicities of the distinct values appearing in $\theta$. If $s_1 \leq d$, then there exists a partition of the index set $\{ 1,\dots, 2d \}$ into $d$ disjoint pairs such that two indices in every pair take different values of $\theta$. Formally, 
    $$
    \{ 1,\dots, 2d \} = \{ j_1, k_1 \} \, \dot{\cup} \cdots \dot{\cup} \, \{ j_d, k_d \}, \quad \theta_{j_r} \neq \theta_{k_r}, \quad r = 1, \dots, d.
    $$
\end{lemma}
\begin{proof}
    Let the distinct values corresponding to the multiplicities $s_1 \geq s_2 \geq \cdots \geq s_t \geq 1$ be given by $a_1, \dots, a_t$. For each $i\in \{1,...,t\}$ we denote by $A_i$ the set of indices from $\{1,...,2d\}$ returning $a_i$,
    $$
    A_i = \{ j \in \{ 1, \dots, 2d \} : \theta_j = a_i \}.
    $$
    Then $|A_i|=s_i \leq d$ for every $i \in \{ 1, \dots, t \}$. Let the elements in $A_i$ be denoted by $a_{i,1}, \dots a_{i,s_i}$. We form a permutation of $\{ 1, \dots, 2d \}$ by concatenating the blocks $A_i$:
    $$
    \left ( \pi(1), \dots, \pi(2d) \right ) = (a_{1,1}, \dots, a_{1,s_1}, a_{2,1}, \dots, a_{2,s_2}, a_{t,1}, \dots, a_{t,s_t}).
    $$
    Hence, in this list, indices with the same $\theta$-value appear consecutively in a block of length $s_i$. Now define for every $r \in \{ 1,\dots, d \}$ the pairs
    $$
    \{ j_r,k_r\} \coloneqq \{ \pi(r), \pi(r+d) \}.
    $$
    The latter pairs form a partition of $\{ 1, \dots, 2d \}$. Moreover, every two elements in each pair have different $\theta$-values: fix $r$ and suppose that $\theta_{\pi(r)} = \theta_{\pi(r+d)} = a_i$. Since all indices with value $a_i$ appear consecutively, the latter implies that $s_i \geq d+1$, which yields a contradiction.
\end{proof}

\begin{remark}
    The previous lemma admits an alternative proof in the language of matching theory: if the sets $A_i$ are defined as in the proof of Lemma \ref{lma:graph_theoretical}, then one can consider them as so-called matching classes attached to $(\theta_1, \dots, \theta_{2d})$. Moreover, one considers the index set $\{ 1, \dots, 2d \}$ as vertices $V$ of a finite undirected graph and defines a complete multipartite graph $K_{s_1,\dots,s_t}$ on the vertex set $V = \{ 1, \dots, 2d \}$ with respect to the classes $A_i$. This graph has the property that $\{u,v \}$ is an edge if and only if $\theta_{u} \neq \theta_{v}$. With this, the existence of the claimed partition in Lemma \ref{lma:graph_theoretical} is equivalent to the existence of a so-called perfect matching in $K_{s_1,\dots,s_t}$. The existence of such a matching is guaranteed by Tutte's 1-factor theorem \cite{lovasz2009matching,plummer1992matching}.
\end{remark}

We are now prepared to prove the first part of Theorem \ref{thm:generic_2d}

\begin{proof}[Proof of Theorem \ref{thm:generic_2d} (1)]

It suffices to prove the statement for $m=2d$. For, if $O$ is a non-empty Zariski-open subset of $\C^{d \times m}$ such that every system $\{ g_j \}_{j=1}^m \simeq (g_1,\dots,g_m) \in O$ does $\Theta$-PR in $\Cd$, then for every $b \in \N$ the set $O \times \C^{d \times b}$ is a non-empty Zariski-open subset of $\C^{d \times (m+b)}$ and every $\{ g_j \}_{j=1}^{m+b} \simeq (g_1,\dots,g_{m+b}) \in O \times \C^{d \times b}$ does $\Theta$-PR in $\Cd$.

\textbf{Step 1: Preliminaries and approach.}
Fix $d \geq 2$, let $m = 2d$ and let $\Theta\subseteq\T$ be finite. Given a choice of $m$ elements from $\Theta$, we identify it with a vector $\theta = (\theta_1, \dots, \theta_{m}) \in \Theta^{m}$. We also denote the vector consisting of the complex-conjugate entries of $\theta$ by $\bar{\theta} \coloneqq (\overline{\theta_1}, \dots, \overline{\theta_{m}})$. Moreover, let $D(\bar{\theta} ) = \mathrm{diag}(\bar{\theta}) \in \C^{m \times m}$ denote the diagonal matrix whose entries are $\bar{\theta}$. For a finite sequence of $m$ vectors $\{ g_j \}_{j=1}^{m} \subseteq \Cd$ we denote by $F$ the matrix in $\C^{d \times m}$ that has columns $g_j$,
$$
F = (g_1, \dots, g_{m}) \in \C^{d \times m},
$$
and let $x_{kl}$ denote the elements of the matrix $F$ (the entry of $F$ in row $k$ and column $l$). Next, for $f,h \in \Cd$ we define a column vector $v_{f,h}$ via
$$
v_{f,h} = \begin{pmatrix}
    \bar{f} \\ \bar{h}
\end{pmatrix} \in \C^{2d}.
$$
Finally, we define the matrix $M(\bar{\theta},F)$ via
$$
M(\bar{\theta},F) \coloneqq \left ( F^{T} , -D(\bar{\theta}) F^T \right ) \in \C^{m \times 2d},
$$
where $F^T$ denotes the transpose of $F$.

Using the notation introduced in this step, we observe that the columns in $F$ do $\Theta$-PR, if and only if for every $f,h\in \C^d$ the following implication is satisfied:
\begin{equation}\label{eq:all_theta_implication}
    M(\bar{\theta}, F ) v_{f,h} = 0 \text{ for some }\theta \in \Theta^m, \text{then } f=\theta_0 h \text{ for some } \theta_0\in \Theta.
\end{equation}
The goal is to show that there exists a non-empty Zariski-open subset $O \subseteq \C^{d \times m} \cong \C^{dm}$ such that every $F \in O$ does $\Theta$-PR. To do so, we first observe that the set $\Theta^{m}$ is finite. Hence, if for every $\theta \in \Theta^{m}$, there exists a non-empty Zariski-open subset $O_\theta$ such that every $F \in O_\theta$ satisfies
\begin{equation}\label{eq:fixed_theta_implication}
    M(\bar{\theta}, F ) v_{f,h} = 0 \ \text{for} \ f,h \in \C^d \implies f=\theta_0 h \ \text{for some} \ \theta_0 \in \Theta,
\end{equation}
then
$$
O = \bigcap_{\theta \in \Theta^{m}} O_\theta
$$
will have the desired property (here we use that a finite intersection of non-empty Zariski-open subsets is non-empty Zariski-open).

For the remainder of the proof, we aim to construct for every $\theta$ a set $O_\theta$ with the above property.

We fix $\theta\in\Theta^m$ and consider two separate cases.

\textbf{Step 2: $\theta$ with at least $d$ repetitions}. Let $\theta =(\theta_1,\dots, \theta_m) \in \Theta^m$ and suppose that $\theta$ has the property that there exist a coordinate that repeats at least $d$ times. Hence, there exists $\theta_0 \in \Theta$ and an index set $I \subseteq \{ 1, \dots, 2d \}$ of cardinality $|I|=d$ such that
$$
\theta_j = \theta_0, \quad j \in I.
$$
Define $O_\theta$ to be the collection of full-spark systems $(g_1,\dots, g_m) \in \C^{d \times m}$. It is known that such a collection forms a non-empty Zariski-open set \cite{alexeev2012full}. Then for every $F \in O_\theta$ and $f,h \in \C^d$, the relation
\begin{equation}
    M(\bar{\theta}, F ) v_{f,h} = 0
\end{equation}
implies that
$$
\langle f-\theta_0h, g_j \rangle = 0, \quad j \in I.
$$
The full-spark property implies that the system $\{ g_j \}_{j \in I}$ is a spanning set for $\Cd$. Consequently, $f=\theta_0h$.

\textbf{Step 3: $\theta$ with less than $d$ repetitions}. We now assume that each of the coordinates of vector $\theta$ repeats less than $d$ times. Define
$$
O_\theta = \{ F \in \C^{d \times m} : M(\bar{\theta},F) \ \text{is invertible} \}.
$$
We note that this question is well-posed, since $M(\bar{\theta},F) \in \C^{m \times 2d} = \C^{2d \times 2d}$.

We have to show three properties of $O_\theta$: i) every $F \in O_\theta$ satisfies \eqref{eq:fixed_theta_implication}; ii) $O_\theta$ is Zariski-open; iii) $O_\theta$ is non-empty.

To show i), observe that if $F \in O_\theta$ then $M(\bar{\theta}, F ) v_{f,h} = 0$ implies that $v_{f,h}=0$. Equivalently, $f=h=0$. In particular $f=\theta_0 h$ for any $\theta_0 \in \Theta$.

We next show ii). Notice that each entry of the matrix $M(\bar{\theta},F)$ is a polynomial in the entries of $F$ (in fact, each entry is a linear function of the entries in $F$). Since composition of polynomials are polynomials, the function $F \mapsto \det M(\bar{\theta},F)$ is a polynomial in the entries of $F$. Re-writing $O_\theta$ as
$$
O_\theta = \{ F \in \C^{d \times m} : \det M(\bar{\theta},F) = 0 \}^c,
$$
implies that $O_\theta$ is Zariski-open.

It remains to show iii), i.e., $O_\theta$ is non-empty. We do this by explicitly constructing a matrix $F$ such that $M(\bar{\theta},F)$ is invertible. To do so, let $s_1 \geq s_2 \geq \cdots \geq s_t$ be the multiplicities of the distinct values appearing in $\theta=(\theta_1, \dots, \theta_{2d})$. Because of the additional assumption on $\theta$, it holds that $s_1\leq d-1$.
By Lemma \ref{lma:graph_theoretical}, there exists a partition
$$
\{ 1,\dots, 2d \} = \{ j_1, k_1 \} \, \dot{\cup} \cdots \dot{\cup} \, \{ j_d, k_d \},
$$
such that $\theta_{j_r} \neq \theta_{k_r}$ for every $r = 1, \dots, d$. We now define a matrix $F$ in terms of the latter partition: for every $r =1,\dots,d$ set
$$
g_{j_r} = g_{k_r} = e_r
$$
where $e_r$ denotes the $r$-th unit vector in $\C^d$. Having defined $F$, we look at the respective matrix $M(\bar{\theta},F)$. Recall, that $M(\bar{\theta},F)$ is given by
$$
M(\bar{\theta},F) = \left ( F^{T} , -D(\bar{\theta}) F^T \right ) = \begin{pmatrix}
    g_1^T & -\overline{\theta_1} g_1^T \\
    \vdots & \cdots \\
    g_{2d}^T & -\overline{\theta_{2d}} g_{2d}^T
\end{pmatrix}.
$$
The determinant of $M(\bar{\theta},F)$ is (up to $\pm 1$ factors) invariant under reordering of the rows. Considering the row-ordering
$$
j_1,k_1,j_2,k_2,\dots, j_d,k_d,
$$
it follows that the determinant of $M(\bar{\theta},F)$ is up to a $\pm 1$ factor equal to the determinant of
$$
\begin{pmatrix}
    1 & 0 & \cdots & \cdots & 0 & -\overline{\theta_{j_1}} & 0 & \cdots & \cdots & 0 \\
    1 & 0 & \cdots & \cdots & 0 & -\overline{\theta_{k_1}} & 0 & \cdots & \cdots & 0 \\
    0 & 1 & 0 & \cdots & 0 & 0 & -\overline{\theta_{j_2}}  & 0 & \cdots & 0 \\
    0 & 1 & 0 & \cdots & 0 & 0 & -\overline{\theta_{k_2}}  & 0 & \cdots & 0 \\
    \vdots & \vdots & \vdots & \vdots & \vdots & \vdots & \vdots  & \vdots & \vdots & \vdots \\
    0 & \cdots & \cdots & 0 & 1 & 0 & \cdots  & \cdots & 0 & -\overline{\theta_{j_d}} \\
    0 & \cdots & \cdots & 0 & 1 & 0 & \cdots  & \cdots & 0 & -\overline{\theta_{k_d}}
\end{pmatrix}.
$$
After an additional re-ordering (this time of the columns), it follows that the determinant of $M(\bar{\theta},F)$ is up to a $\pm 1$ factor equal to the determinant of the block-diagonal matrix
$$
\begin{pmatrix}
    1 & -\overline{\theta_{j_1}} & & & \\
    1 & -\overline{\theta_{k_1}} & & & \\
    & & 1 & -\overline{\theta_{j_2}} & \\
    & & 1 & -\overline{\theta_{k_2}} & \\
    & & & & \ddots \\
    & & & & & 1 & -\overline{\theta_{j_d}}\\
    & & & & & 1 & -\overline{\theta_{k_d}}
\end{pmatrix}.
$$
The determinant of this block-diagonal matrix is equal to the complex conjugate of the product
$$
\prod_{r=1}^d \left ( \theta_{j_r} - \theta_{k_r} \right ).
$$
Since $\theta_{j_r} \neq \theta_{k_r}$ for every $r = 1, \dots, d$, it follows that this product does not vanish. We have therefore constructed a matrix $F$ such that $\det M(\bar{\theta},F) \neq 0$. Consequently, $O_\theta \neq \varnothing$ and part iii) is proved.
\end{proof}

Next, we prove the second part of Theorem \ref{thm:generic_2d}.

\begin{proof}[Proof of Theorem \ref{thm:generic_2d} (2)]
    Let $\{ \theta_j : j \in \N \}$ be an enumeration of the elements in $\Theta$. Further, define the subset consisting of the first $n$ elements of $\Theta$ via
    $$
    \Theta_n \coloneqq \{ \theta_1, \dots, \theta_n \}.
    $$
    Since $\Theta_n$ is finite, the first part of Theorem \ref{thm:generic_2d} implies that there exists a non-empty Zariski-open subset $O_n \subseteq \C^{d \times m}$ such that every $\{g_j\}_{j=1}^{m} \cong (g_1,\dots, g_{m}) \in O_n$ does $\Theta_n$-PR. Since every non-empty Zariski-open subset of $\C^N$ with $N \in \N$ is open and dense in the Euclidean topology, it follows from Baire's category theorem that the intersection
    $$
    \mathcal{B} = \bigcap_{n \in \N} O_n
    $$
    is a $G_\delta$-set that is still dense in Euclidean topology. In particular, $\mathcal{B}$ is non-empty. Furthermore $\mathcal{B}$ has full measure. Indeed, let $\mu$ denote the Lebesgue measure on $\C^{d \times m} \cong \C^{dm}$. Since for every $n\in \N$ we have $\mu(O_n^c)=0$, it follows that
    $$\mu(\mathcal{B}^c) = \mu(\bigcup_{n\in \N} O_n^c) \leq \sum_{n=1}^\infty \mu(O_n^c) = 0.$$
    
    Now let $\{ g_j \}_{j=1}^{m} \in \mathcal{B}$ and suppose that $f,h \in \C^d$ are such that
    $$
    \langle f,g_j \rangle = \theta_j \langle h,g_j \rangle, \quad j = 1,\dots,m,
    $$
    for some $\theta_1, \dots, \theta_{m} \in \Theta$. For large enough $n$, it holds that $\theta_1, \dots, \theta_{m} \in \Theta_n$. Since $\{ g_j \}_{j=1}^{m}$ lies in the intersection of all $O_n$, it does $\Theta_n$-PR. Consequently, $f=\theta_0 h$ for some $\theta_0 \in \Theta_n \subseteq \Theta$, which shows that $\{ g_j \}_{j=1}^{m}$ does $\Theta$-PR.
\end{proof}

Finally, combining Theorem \ref{thm:failure} with Theorem \ref{thm:generic_2d}, we obtain Corollary \ref{cor:minimality_theta_pr}.

\subsection{$\Theta$-PR in $\C^d$ for an uncountable $\Theta$}\label{sec:uncountable}

The goal for this section is to prove Proposition \ref{prop:arc}. To do so, we start with some preliminary definitions.

First, we recall that the Caley transform is the Möbius transform
$$
C(z) \colonequals i\,\frac{1+z}{1-z}, 
$$
with inverse
$$
C^{-1}(w) \colonequals \frac{w-i}{w+i}.
$$
The map $C$ is a biholomorphism from the open unit disc $\mathbb D \colonequals \{z\in\C : |z|<1\}$ onto the upper half-plane $\mathbb H \colonequals \{w\in\C : \im w > 0\}$.
On the boundary $\T$ of $\mathbb D$, the Caley transform extends continuously to a homeomorphism
$$
C : \T \to \hat \R, \qquad C(1) = \infty,
$$
sending $\T$ onto the extended real line $\hat \R = \R \cup \{ \infty \}$.
A direct computation shows that for $t\in(0,2\pi)$,
\begin{equation}\label{eq:cayley-on-circle}
  C(e^{it})
  = i\,\frac{1+e^{it}}{1-e^{it}} = - \cot\Bigl(\frac{t}{2}\Bigr)
  = \cot\Bigl(-\frac{t}{2}\Bigr).
\end{equation}
Thus the map $t \mapsto C(e^{it})$ is strictly increasing on $(0,2\pi)$ and maps this interval bijectively onto $\R$. In particular, any closed arc of $\T$ is carried by $C$ onto a closed interval in $\hat \R$.
Observe further that for any $b \in \R$ the horizontal translation
$$
T_b(w) \colonequals w + b
$$
is an automorphism of $\mathbb H$ preserving $\hat \R$, and therefore
$$
M(z) \colonequals C^{-1}\bigl(C(z) + b\bigr)
$$
yields an automorphism of $\mathbb D$ whose boundary values give a Möbius automorphism of $\T$, that is, $M \in \mathrm{Aut}(\T)$.

We are prepared to prove the following lemma.

\begin{lemma}\label{lma:arc_bijection}
    Let
    $$
      A \colonequals \{ e^{it} : t \in [\alpha, \alpha + \ell] \}, 
      \qquad 
      \mathcal{A} \colonequals \{ e^{it} : t \in [\beta, \beta + L] \},
    $$
    be two arcs on $\T$ with $\alpha,\beta \in \R$ and $\ell,L \in (0,2\pi)$. Then there exists a Möbius transform $M\in\mathrm{Aut}(\T)$ such that $M(A)=\mathcal{A}$.
\end{lemma}
\begin{proof}
First we consider a special case when $\alpha=0=\beta$.
Using equation \eqref{eq:cayley-on-circle}, for $t\in(0,\ell]$ we have
$$
  C(e^{it}) = -\cot\Big(\frac{t}{2}\Big) \in \R.
$$
As $t$ increases from $0$ to $\ell$, this expression increases from $-\infty$ to $-\cot(\ell/2)$. Together with $C(1)=\infty$, this shows that
$$
  C(A)
  = \left \{ -\cot\left (\frac{t}{2} \right ) : t\in(0,\ell] \right \} \cup \{\infty\}
$$
is a closed interval in $\hat \R$ with finite endpoint $-\cot(\ell/2)$ and infinite endpoint $\infty$.
Now let $b\in\R$ be such that
$$
  -\cot\Big(\frac{\ell}{2}\Big) + b = -\cot\Big(\frac{L}{2}\Big),
$$
The translation $T_b(w) = w + b$ is an automorphism of $\mathbb H$ preserving $\hat \R$, and it maps the finite endpoint of $C(A)$ to the finite endpoint of
$$
  C(\mathcal{A})
  = \left \{ -\cot\left (\frac{t}{2} \right ) : t\in(0,L] \right \} \cup \{\infty\}.
$$
We therefore obtain
$$
  T_b\bigl(C(A)\bigr)
  = C(\mathcal{A}).
$$
Now define
$$
  F(z) \colonequals C^{-1}\bigl(C(z) + b\bigr).
$$
Then $F\in \text{Aut}(\T)$ is a Möbius transform that satisfies $F\bigl(A\bigr)=\mathcal{A}$.

Consider now the general case of any $\alpha,\beta\in \R$. The rotation $z\mapsto e^{i\beta}z$ maps $\{e^{it}:t\in[0,L]\}$ onto $\mathcal A$, and the rotation $z\mapsto e^{-i\alpha}z$ maps $A$ onto $\{e^{it}: t\in [0,\ell]\}$. Therefore, using the Möbius transform $F\in \text{Aut}(\T)$ from the special case, the composition
$$
  M(z) \colonequals e^{i\beta}\,F(e^{-i\alpha}z)
$$
is a Möbius transform in $\mathrm{Aut}(\T)$ that satisfies
$
  M(A) = \mathcal A.
$
\end{proof}

The previous lemma in combination with Theorem \ref{thm:mbs} implies that if $\Theta$ and $\Theta'$ are two non-trivial arcs then $\mathbf{G}$ does $\Theta$-PR if and only if $\mathbf{G}$ does $\Theta'$-PR. In the special case of finite dimensional Hilbert spaces we have the following statement.

\begin{proposition}
Suppose that $\Theta \subseteq \T$ contains a non-trivial arc. Then $\mathbf{G}=\{ g_j \}_{j=1}^m \subseteq \Cd$ does $\Theta$-PR if and only if it does PR. In particular,
$$
\mathcal{N}(\C^d,\Theta) \geq 4d -4 - 2\log_2(d).
$$
\end{proposition}
\begin{proof}
    From Lemma \ref{lemma:theta_pr_implies_theta'_pr}, if $\mathbf{G}$ does PR then it does $\Theta$-PR. It remains to show the reverse direction. To do so, we show that if $\mathbf{G}$ fails PR then it also fails $\Theta$-PR.

    Thus, suppose that $\mathbf{G}$ fails PR. Hence, by Lemma \ref{lma:failure_witnessing}, there exist linearly independent $f,h\in \C^d$ and $\Theta'=\{\theta_1',...,\theta_m'\}\subseteq \T$ such that for every $j \in \{1, \dots, m \}$ one has
    $$\langle f,g_j\rangle = \theta'_j \langle h,g_j\rangle.$$   
    Since $\Theta'$ is a finite subset of $\T$, there exists a proper arc $A$ of $\T$ with $\Theta'\subseteq A$. From Lemma \ref{lemma:theta_pr_implies_theta'_pr} we obtain that $\mathbf{G}$ fails $A$-PR. Next, according to Lemma \ref{lma:arc_bijection}, there exists a Möbius transform $M \in \mathrm{Aut}(\T)$ such that $M(A)=\mathcal{A}$. From Theorem \ref{thm:mbs} this gives that $\mathbf{G}$ fails $\mathcal{A}$-PR, thus again from Lemma \ref{lemma:theta_pr_implies_theta'_pr} it implies that $\mathbf{G}$ fails $\Theta$-PR.

    To prove the second part of the statement, we observe that the first part implies 
    \begin{equation}\label{eq:equality_Theta_T}
        \mathcal{N}(\C^d,\Theta)  = \mathcal{N}(\C^d,\T).
    \end{equation}
    According to \cite{heinosaari2013quantum}, it holds that
    \begin{equation}\label{eq:heinn}
        \mathcal{N}(\C^d,\T) \geq \begin{cases}
            4d-4-2\alpha(d) + 2, & d \ \text{odd, and} \ \alpha(d) = 2 \mod 4 \\
            4d-4-2\alpha(d) + 3, & d \ \text{odd, and} \ \alpha(d) = 3 \mod 4 \\
            4d-4-2\alpha(d) + 1, & \text{else},
        \end{cases},
    \end{equation}
    where $\alpha(d)$ denotes the number of ones in the binary expansion of $d-1$. The quantity $\alpha(d)$ satisfies the inequality $\alpha(d) \leq \log_2(d)$. Inserting this estimate into \eqref{eq:heinn} and using identity \eqref{eq:equality_Theta_T} yields the statement.
    
\end{proof}

% \section{Open problems}

% \subsection{} Characterization for uncountable $\Theta$

% \subsection{} Minimality between $2d$ and $4d$ ...

% \subsection{} Irregular sampling

% \subsection{} Shift-invariant spaces

\section*{Acknowledgments}

The authors would like to thank the organizers of the research term \textit{Lattice Structures in Analysis and Applications} at the Instituto de Ciencias Matemáticas (Madrid), as well as the organizers of the conference \textit{Phase Retrieval and Banach Lattices} at ETH Zurich, for their hospitality, which made this collaboration possible.

Lukas Liehr is grateful to the Azrieli Foundation for the award of an Azrieli Fellowship, which supported this research. 

\bibliographystyle{plain}
\bibliography{bibfile}

\end{document}